\documentclass[11pt]{article}
\usepackage{amsmath,amsfonts,amsthm,amscd,amssymb,graphicx,mathscinet, color}

\usepackage[utf8]{inputenc}
\usepackage{amsbsy}
\usepackage{url}
\usepackage{graphicx}
\usepackage[margin=1.25in,centering]{geometry}
\newtheorem{Lemma}{Lemma}[section]

\newtheorem{Hypothesis}[Lemma]{Hypothesis}
\newtheorem{Proposition}[Lemma]{Proposition}
\newtheorem{Remark}[Lemma]{Remark}
\newtheorem{Theorem}{Theorem}

\newenvironment{Proof}[1][.]%
 {\begin{trivlist}\item[]\textbf{Proof#1 }}%
 {\hspace*{\fill}$\rule{0.3\baselineskip}{0.35\baselineskip}$\end{trivlist}}

 {\begin{trivlist}\item[]\textbf{Acknowledgments }}{\end{trivlist}}

\makeatletter\@addtoreset{equation}{section}\makeatother

\newcommand{\cA}{\mathcal{A}}    
\newcommand{\cD}{\mathcal{D}}   
\newcommand{\cE}{\mathcal{E}}    
\newcommand{\cH}{\mathcal{H}}    
\newcommand{\cM}{\mathcal{M}}    
\newcommand{\cT}{\mathcal{T}}    
\newcommand{\cU}{\mathcal{U}}    
\newcommand{\cV}{\mathcal{V}}    
\newcommand{\cW}{\mathcal{W}}

\newcommand{\LL}{\vartheta}    

\newcommand{\bE}{\mathbb{E}}

\newcommand{\R}{\mathbb{R}}             % reals
\newcommand{\Z}{\mathbb{Z}}             % integers

\newcommand{\rmd}{\mathrm{d}}           % derivatives
\begin{document}

\title{Validated spectral stability via conjugate points}

\author{Margaret Beck\footnotemark[1]  \and Jonathan Jaquette\footnotemark[2] }

\date{\today}
\maketitle

\renewcommand{\thefootnote}{\fnsymbol{footnote}}

\footnotetext[1]{Department of Mathematics and Statistics, 
Boston University; mabeck@bu.edu.}

\footnotetext[2]{Department of Mathematics and Statistics, 
Boston University; jaquette@bu.edu.}

\begin{abstract}
Classical results from Sturm-Liouville theory state that the number of unstable eigenvalues of a scalar, second-order linear operator is equal to the number of associated conjugate points. Recent work has extended these results to a much more general setting, thus allowing for spectral stability of nonlinear waves in a variety of contexts to be determined by counting conjugate points. However, in practice, it is not yet clear whether it is easier to compute conjugate points than to just directly count unstable eigenvalues. We address this issue by developing a framework for the computation of conjugate points using validated numerics. Moreover, we apply our method to a parameter-dependent system of bistable equations and show that there exist both stable and unstable standing fronts. This application can be seen as complimentary to the classical result via Sturm-Louiville theory that in scalar reaction-diffusion equations pulses are unstable whereas fronts are stable, and to the more recent result of ``Instability of pulses in gradient reaction-diffusion systems: a symplectic approach," by Beck et. al., that symmetric pulses in reaction-diffusion systems with gradient nonlinearity are also necessarily unstable. 
\end{abstract}

%%%%%%%%%%%%%%%%%%%%%%%%%%%%%%%%%%%%%%%%%%%%%%%%%%%%%%%%%%%%%%%%%%%%%%%%%%%%

\section{Introduction}

Understanding the stability of solutions is important for predicting the long-time behavior of partial differential equation models of physical systems. A key step is analyzing spectral stability, which means determining whether the linearization $\mathcal{L}$ of the PDE about a given solution of interest has any spectral elements with positive real part. In many cases, it is useful to divide the spectrum of $\mathcal{L}$ into two disjoint subsets, the essential and point spectrum: $\sigma(\mathcal{L}) = \sigma_\mathrm{ess}(\mathcal{L}) \cup \sigma_\mathrm{pt}(\mathcal{L})$. The essential spectrum is often relatively easy to compute, so determining spectral stability amounts to looking for unstable eigenvalues \cite{KapitulaPromislow13}. 

Scalar, second-order eigenvalue problems can be understood by classical Sturm-Liouville theory. A simple such problem is of the form
\begin{equation}\label{E:sl}
\mathcal{L} u := u_{xx} + q(x) u = \lambda u, \qquad x \in [a,b], \qquad u(a) = u(b) = 0, \qquad \lambda \in \mathbb{R},
\end{equation}
with $q(x)$ smooth, but the following discussion is applicable to a much larger class of equations and boundary conditions \cite{BirkhoffRota89}. Using Pr\"ufer coordinates, which for this equation are just the polar coordinates 
\[
u(x; \lambda) = r(x; \lambda) \sin [\theta(x; \lambda)], \qquad u_x (x; \lambda) = r(x; \lambda) \cos [\theta(x; \lambda)],
\]
one obtains a system for $(r, \theta)$ given by
\begin{equation}\label{E:rtheta}
r' = r(1 + \lambda - q(x)) \cos \theta \sin \theta, \qquad \theta' = \cos^2\theta + (q(x) - \lambda)\sin^2\theta.
\end{equation}
Eigenvalues are values of $\lambda$ for which a solution of \eqref{E:sl} exists. In \eqref{E:rtheta}, $\{r = 0\}$ is invariant, the $\theta$ equation decouples, and so the existence of an eigenvalue can be reformulated as follows. Consider the initial condition $\theta(a; \lambda) = 0$ and flow forward to $x = b$; $\lambda$ is an eigenvalue if and only if $\theta(b; \lambda) \in \{j\pi\}_{j\in\mathbb{Z}}$. The form of \eqref{E:rtheta} implies that, if $\lambda$ is negative and sufficiently large, then $\theta' > 0$ and $\theta$ will be forced to oscillate. Thus, there exists at least one eigenvalue, which we label $\lambda_k$, such that $\theta(b; \lambda_k) = (k+1)\pi$. Increasing $\lambda$ reduces the oscillation in $\lambda$, and since $\theta(b; \lambda)$ varies continuously in $\lambda$, we obtain an increasing sequence of eigenvalues $\{\lambda_j\}_{j=0}^\infty$ such that $\theta(b; \lambda_j) = (j+1)\pi$. For $\lambda > \lambda_0$, $\theta$ no longer has enough ``time" to complete one (half) oscillation, and so no more eigenvalues exist. One obtains a result regarding the sequence of eigenvalues $\lambda_0 > \lambda_1 > \dots$ and corresponding eigenfunctions $u_0, u_1, \dots$ that says that $u_k$ has exactly $k$ zeros. 

From a stability perspective, this result is very powerful, as can be seen by the following example. Suppose $\mathcal{L}$ is the linearization of a scalar reaction-diffusion equation, $v_t = v_{xx} + f(v)$, about a stationary solution $\varphi$. Then $\mathcal{L} \varphi ' = 0$. In other words, $u = \varphi'$ is the eigenfunction associated with eigenvalue zero. If $\varphi$ is a pulse, then $\varphi'$ has exactly one zero. Hence, $\varphi' = u_1$, $\lambda_1 = 0$, $\lambda_0 > 0$ and an unstable eigenvalue exists. On the other hand, if $\varphi$ is a monotonic front, then $\lambda_0 = 0$ and no unstable eigenvalues exist. It is worth noting that this (in)stability result is forced by the topology, and it does not depend on the details of the function $f$ or the underlying solution $\varphi$. 

This problem can be reformulated in terms of conjugate points. Instead of \eqref{E:sl}, consider the same problem but on a variable domain with $\lambda = \lambda_*$ fixed,
\begin{equation}\label{E:sl-conj}
u_{xx} + q(x) u = \lambda_* u, \qquad x \in [a,s], \qquad u(a) = u(s) = 0, \qquad a \leq s \leq b.
\end{equation}
A conjugate point is a value of $s$ such that the above equation has a solution. Equation \eqref{E:rtheta} is still relevant, with the value of $s$ controlling how much ``time" $\theta$ has to oscillate. If we choose $\lambda_* = \lambda_k$, we find $s_k := b$ is a conjugate point. Decreasing $s$ produces a sequence of conjugate points $\{s_j\}_{j=0}^k$, with $\theta(s_j; \lambda_k) = (j+1)\pi$. This leads to a relationship between conjugate points and eigenvalues: the number of conjugate points that exist for some fixed $\lambda_*$ is equal to the number of eigenvalues $\lambda$ such that $\lambda > \lambda_*$. Choosing $\lambda_* = 0$ allows one to count unstable eigenvalues by counting conjugate points. In the scalar case, counting conjugate points is equivalent to counting the number of zeros of the eigenfunction associated with the zero eigenvalue. 

Because the above perspective relies on the reformulation of the eigenvalue problem in terms of polar coordinates, it is not immediately clear how to generalize this to systems of equations, ie an eigenvalue problem $\mathcal{L} u = \lambda u$ with $u \in \mathbb{R}^n$ and $n > 1$. It turns out that this generalization is provided by a topological invariant called the Maslov index \cite{Arnold67,Arnold85, Bott56}. Given a symplectic form $\omega$ on $\mathbb{R}^{2n}$, ie a nondegenerate, bilinear, skew-symmetric form, the set of Lagrangian subspaces $\ell$ in $\mathbb{R}^{2n}$ with respect to $\omega$ is defined by
\begin{equation}\label{E:ln}
\Lambda(n) = \{ \ell \subset \mathbb{R}^{2n}: \mathrm{dim}(\ell) = n, \quad \omega(U, W) = 0 \quad \forall \quad U, W \in \ell \}. 
\end{equation}
It turns out that the fundamental group of $\Lambda(n)$ is the integers, $\pi_1(\Lambda(n)) = \mathbb{Z}$. This allows one to define a phase in $\Lambda(n)$, similar to the angle $\theta$ above, and the Maslov index counts the winding of this phase. 

There are many recent results that use the Maslov index to prove that one can count unstable eigenvalues by counting conjugate points. See, for example, \cite{BairdCornwallCox20, BeckCoxJones18, CoxJonesLatushkin16, CornwallJones20, CoxJonesMarzuola15, DengJones11, Howard20, HowardLatushkinSukhtayev18, JonesLatushkinMarangell13}. The exact statements of the results and their proofs depend on the details of the situation being considered. Notably, some of the results are valid on spatial domains $\Omega \subset \mathbb{R}^d$ with $d \geq 1$ \cite{CoxJonesLatushkin16,CoxJonesMarzuola15, DengJones11}. The vast majority of techniques concerning the stability of nonlinear waves are valid primarily in one spatial domain, and so results in higher space dimensions are of particular importance. In fact, it was the multidimensional result \cite{DengJones11} that inspired the recent extensive activity in using the Maslov index to study stability. Although these results are very interesting in their own right, if it is just as difficult to count conjugate points as it is to count unstable eigenvalues, then from a stability perspective not much has been gained. This work addresses exactly this issue.

The primary tool for studying spectral stability of nonlinear waves is arguably the Evans function \cite{Sandstede02}. It has been used extensively in one-dimensional domains, and in multidimensional domains with a distinguished direction, such as a channel $\mathbb{R} \times \Omega \subset \mathbb{R} \times \mathbb{R}^{d-1}$ with $\Omega$ compact \cite{OhSandstede10}. Although there are some results in more general multidimensional domains \cite{DengNii08}, it is not clear how useful they are in practice. This is one key reason why the possibility of using the Maslov index to study stability, in particular for multidimensional problems, is so interesting. 

There are very few results where the Maslov index has been used to analyze stability in a context where the Evans function appears to not be similarly applicable, and all are in one spatial dimension. For example, \cite{ChenHu14} applies to a specific FitzHugh-Nagumo model, and it relies heavily on the activator-inhibitor structure present there. Another, \cite{BeckCoxJones18}, applies to systems of reaction-diffusion equations with gradient nonlinearity. This latter result is used to show that symmetric pulse solutions of such systems are necessarily unstable, which is the generalization of the above-mentioned result from Sturm-Liouville theory to the system case. 

The main goal of this paper is to provide a framework for rigorously counting conjugate points using validated numerics in the setting of \cite{BeckCoxJones18}. This will allow for the Maslov index to be used to determine stability in more applications. Furthermore, we argue that, when applicable, this a more efficient method for determining stability than the Evans function. This is because counting unstable eigenvalues using the Evans function requires computing its winding number as $\lambda$ varies in a contour in the complex plane, whereas counting conjugate points requires computation only for one fixed value $\lambda = 0$. For comparison, there appear to be only two instances of validated numerics for the Evans function, and in those works stability calculations are reported to take 425 hours \cite{ArioliKoch15} and 4 hours \cite{BarkerZumbrun16}, respectively.  Our method for rigorously counting conjugate points currently takes less than half a minute, a significant improvement. It may be possible to speed up Evans function computations, eg using \cite{BarkerNguyenSandstede18}, but this has not yet been done.  
For further reading on the validated computation of stability in PDEs we refer to the book  \cite{nakao2019numerical} and references cited therein.

The general setting in which our method will apply is that of  \cite{BeckCoxJones18}, which we now present. Consider an eigenvalue problem
\begin{equation}\label{E:eval}
\lambda u = Du_{xx} + \mathcal{M}(x) u =: \mathcal{L} u, \qquad u \in \mathbb{R}^n, \qquad x \in \mathbb{R}
\end{equation}
where $D$ is a diagonal matrix with positive entries and $\mathcal{M}$ satisfies Hypothesis \ref{H:potential}. 
\begin{Hypothesis} \label{H:potential} The matrix $\mathcal{M}(x)$ satisfies
\begin{list}{}{\itemsep 0.5mm \topsep 0.5mm \itemindent 0mm}
\item (H1)  $\mathcal{M} \in C(\mathbb{R}, \mathbb{R}^{n\times n})$ and is for each $x$ a symmetric matrix with real entries. 
\item (H2) The limits $\lim_{x\to\pm\infty}\mathcal{M}(x) =: \mathcal{M}_\pm$ exist, are negative definite, and have distinct eigenvalues. 
\item (H3) There exist constants $C_M$ and $\eta_M$ such that $|\mathcal{M}(x) - \mathcal{M}_\pm| \leq C_M e^{-\eta_M|x|}$ for $x \in \mathbb{R}^\pm$, respectively.  
\end{list}
\end{Hypothesis}
Moreover, we require that the operator $\mathcal{L}$ has a simple eigenvalue at zero.
\begin{Hypothesis}\label{H:simple}
The operator $\mathcal{L}$, defined on $X = L^2(\mathbb{R})$ with domain $Y = H^2(\mathbb{R})$, satisfies
\begin{list}{}{\itemsep 0.5mm \topsep 0.5mm \itemindent 0mm}
\item (H4) $\lambda = 0$ is an eigenvalue of geometric and algebraic multiplicity one.
\end{list}
\end{Hypothesis}
We will study the spectrum of $\mathcal{L}$ on $X = L^2(\mathbb{R})$ with domain $Y = H^2(\mathbb{R})$. Hypothesis (H1) implies that $\mathcal{L}$ is self adjoint, and so we can take $\lambda \in \mathbb{R}$. This is necessary for the symplectic framework used below. If, for example, $\mathcal{L}$ is the linearization of a system of reaction-diffusion equations with gradient nonlinearity,
\begin{equation}\label{E:gradrd}
v_t = Dv_{xx} + \nabla G(v), \qquad G \in C^2(\mathbb{R}^n, \mathbb{R}),
\end{equation}
about a stationary solution $\varphi$, then (H1) holds. The assumption in (H2) that $\mathcal{M}_\pm$ are negative definite is necessary and sufficient for ensuring that the essential spectrum of $\mathcal{L}$ lies in the open left half plane and is bounded away from the imaginary axis. This is a natural assumption; if the essential spectrum is unstable, then stability has already been determined. If the essential spectrum is only marginally stable, meaning it lies in the left half plane but touches the imaginary axis, this can also be handled using the modification in \cite[Remark 1.3]{BeckCoxJones18}, but we do not consider this further here.  
 Assumption (H3) appears at first glance to be stronger than the corresponding hypothesis stated in \cite{BeckCoxJones18}, where it is only required that the functions $x \to \mathcal{M}(x) - \mathcal{M}_\pm$ be in $L^1(\mathbb{R}^\pm)$, which implies that the associated first order eigenvalue problem has exponential dichotomies on the half lines. However, for the primary setting of interest here, \eqref{E:gradrd}, the fact that $\mathcal{M}_\pm = \nabla^2G(\varphi_\pm) < 0$ implies that the underlying solution $\varphi$ has constant limits $\varphi_\pm = \lim_{x \to \pm \infty}\varphi(x)$ that are approached exponentially fast, which in turn implies (H3). It should be possible to extend our method to the $L^1$ case by adjusting some of the estimates for the constants $L_\pm$ that appear below. Hypothesis (H4) did not appear in \cite{BeckCoxJones18}. If $\varphi$ is a stationary solution to \eqref{E:gradrd}, then $\varphi'$ will be an eigenfunction with eigenvalue zero. Hence, part of (H4) is natural. The additional assumption that zero is a simple eigenvalue plays a crucial role below, particularly in the determination of $L_+$ in \S \ref{S:L_+}. It is likely possible to adapt our method to cases with higher multiplicity, but we have not yet explored that further. In particular, at the moment it is not clear the extent to which our method could be used to study bifurcations. 

The eigenvalue problem \eqref{E:eval} can be written as a first order eigenvalue problem of the form 
\begin{equation}\label{E:eval-sys}
U_x = J \mathcal{B}(x; \lambda) U, \qquad U \in \mathbb{R}^{2n}, \qquad J = \begin{pmatrix} 0 & -I_n \\ I_n & 0 \end{pmatrix}, \qquad \lambda \in \mathbb{R},
\end{equation}
where $U = (u, D u_x)$ and 
\begin{equation}\label{E:defB}
\mathcal{B}(x; \lambda) = \begin{pmatrix} \lambda - \mathcal{M}(x)  & 0 \\ 0 & -D^{-1} \end{pmatrix}
\end{equation}
is symmetric. Consider the canonical symplectic form on $\mathbb{R}^{2n}$ defined by
\begin{equation}\label{E:omega-J}
\omega: \mathbb{R}^{2n} \times \mathbb{R}^{2n} \to \mathbb{R}, \qquad \omega(U, W) = \langle U, JW \rangle, 
\end{equation}
where $\langle \cdot, \cdot \rangle$ is the usual inner product in $\mathbb{R}^{2n}$. 
We now connect the evolution of \eqref{E:eval-sys} with the set of Lagrangian planes defined in \eqref{E:ln}.

Assumption (H2) implies that the limits $\lim_{x \to \pm \infty} J\mathcal{B}(x; \lambda) = J \mathcal{B}_\pm(\lambda)$ exist and are hyperbolic. In order for $U$ to correspond with an eigenfunction, we therefore need $|U(x; \lambda)| \to 0$ as $x \to \pm \infty$. Let $\mathbb{E}_-^u(x; \lambda)$ denote the subspace of solutions that is asymptotic, as $x \to -\infty$, to the unstable subspace of $J \mathcal{B}_-(\lambda)$. This is exactly the subspace of solutions that decay as $x \to -\infty$. It turns out that this subspace is Lagrangian: $\mathbb{E}_-^u(x; \lambda) \in \Lambda(n)$ for each $(x, \lambda) \in \mathbb{R}^2$. To see this, notice that if $U, W \in \mathbb{E}_-^u(x; \lambda)$, then
\[
\frac{d}{dx} \omega(U(x), W(x)) = \langle U', JW \rangle + \langle U, JW' \rangle = \langle J\mathcal{B}U, JW \rangle + \langle U, J^2\mathcal{B}W \rangle = 0,
\]
since $J^{-1} = J^* = -J$ and $\mathcal{B}^* = \mathcal{B}$. Since $U(x), W(x) \to 0$ as $x \to -\infty$, we see $\omega(U(x), W(x)) = 0$ for all $x$. The fact that $\mathrm{dim}(\mathbb{E}_-^u) = n$ follows from assumptions (H1) and (H2); see the proof of Lemma \ref{lem:no-asym-conj} below. The subspace known as the Dirichlet subspace is also Lagrangian,
\begin{equation}\label{E:dirichlet}
\mathcal{D} = \{ (u, w) \in \mathbb{R}^{2n}: u = 0\}, \qquad \mathcal{D} \in \Lambda(n).
\end{equation}
For fixed $\lambda = 0$, conjugate points are defined to be values of $s$ such that there is a nontrivial intersection between these two Lagrangian planes: $\mathbb{E}_-^u(s; 0) \cap \mathcal{D} \neq \{0\}$. See \cite[Def 2.7]{BeckCoxJones18}. 

\begin{Theorem}\label{thm:prev-thm}\cite[Thm 3.1, Thm 4.1]{BeckCoxJones18}. The number of positive eigenvalues of $\mathcal{L}$, defined in \eqref{E:eval}, is equal to the number of conjugate points, ie the number of values $s \in \mathbb{R}$ such that $\mathbb{E}_-^u(s; 0) \cap \mathcal{D} \neq \{0\}$. Moreover, there exists an $L_+^*$ sufficiently large such that, for any $L_+ > L_+^*$, the number of conjugate points in $(-\infty, \infty)$ is equal to the number of conjugate points in $(-\infty, L_+)$. 
\end{Theorem}

This result is proven in \cite{BeckCoxJones18} using the Maslov index. However, we will not need to use the Maslov index here. Therefore, we refer to \cite{BeckCoxJones18} for a definition of the Maslov index, as well as for the proof of the above Theorem. 

The goal of this work is to provide a method for rigorously computing the number of conjugate points on $(-\infty, L_+)$. To do so, there are several issues that must be overcome. First, since computations will be done numerically, they can only be done on a finite interval, $(-L_-, L_+)$. The existence of the upper limit $L_+$ is given by Theorem \ref{thm:prev-thm}, but we do not know how big sufficiently large is. We must therefore show that is possible to choose a finite value $L_-$, and then prescribe a way for choosing explicit values $L_\pm$ so that we capture all of the conjugate points. In other words, we need to know when to start and stop the computation. 

Once this has been accomplished, we need an algorithm for detecting conjugate points on $(-L_-, L_+)$. This will be accomplished via the fame matrix associated with $\mathbb{E}_-^u(x; 0)$. More specifically, let 
\[
\begin{pmatrix} A_1(x) \\ A_2(x) \end{pmatrix}, \qquad A_1(x), A_2(x) \in \mathbb{R}^{n \times n}
\]
be a matrix whose columns form a basis for $\mathbb{E}_-^u(x; 0)$. We have the following Lemma, which is well known. 

\begin{Lemma}\label{lem:conj-det}
$\mathbb{E}_-^u(x; 0) \cap \mathcal{D} \neq \{0\}$ if and only if $\mathrm{det}A_1(x) = 0$.
\end{Lemma}

\begin{Proof}
If $\mathbb{E}_-^u(x; 0) \cap \mathcal{D} \neq \{0\}$ then there exists a nonzero vector $u \in \mathbb{R}^n$ such that 
\[
\begin{pmatrix} A_1(x) \\ A_2(x) \end{pmatrix} u = \begin{pmatrix} 0 \\ \tilde u \end{pmatrix}.
\]
This implies that $A_1(x) u = 0$, and so $\mathrm{det}A_1(x) = 0$. On the other hand, if $\mathrm{det}A_1(x) = 0$, then there exists a nonzero $u \in \mathbb{R}^n$ such that $A_1(x) u = 0$.
As $\mathbb{E}_-^u(x; 0)$ is $n$ dimensional its frame matrix is of full rank, having a trivial kernel. 
That is to say $ \tilde{u} \neq 0$ in the above equation, whereby $\mathbb{E}_-^u(x; 0) \cap \mathcal{D} \neq \{0\}$. 
\end{Proof}

Thus, once we have the frame matrix for $\mathbb{E}_-^u(x; 0)$, we simply need to locate values of $s \in [-L_-, L_+]$ such that $\mathrm{det}A_1(s) = 0$. In other words, we will have reduced the determination of stability to the computation of zeros of the scalar valued function $\mathrm{det}A_1(x)$ on the finite interval $[-L_-, L_+]$. We note that any zeros of this function will be isolated; see \cite[Remark 3.5]{BeckCoxJones18}.

We now summarize our numerical methodology for computing the spectral stability of stationary solutions. First, one must compute a stationary solution $\varphi  $ to the PDE \eqref{E:gradrd}, which is then used to define the eigenvalue problem in \eqref{E:eval-sys}.  
After fixing $L_-,L_+\geq 0$,  one next computes a numerically approximate frame matrix for $\mathbb{E}_-^u(x; 0)$ for $ x \in [-L_-,L_+]$. 
To do so,  first we define an approximate frame matrix for $\mathbb{E}^\mathrm{u}_-(-L_-;0)$ by taking as an initial value a matrix $A(-L_-)$ whose columns are the $n$-unstable eigenvectors of $J \mathcal{B}_-(0)$. 
Then for $-L_- \leq  x \leq L_+$, the frame matrix $A(x)= \left( \begin{smallmatrix} A_1(x) \\ A_2(x) \end{smallmatrix} \right)$ is defined by numerically integrating each of the columns forward in time according to the differential equation in \eqref{E:eval-sys}. 
Lastly, one counts the number of times  
the scalar valued function $\mathrm{det}A_1(x)$ equals zero on the finite interval $x \in [-L_-, L_+]$, which by Lemma \ref{lem:conj-det} is equal to the number of conjugate points, and hence by Theorem \ref{thm:prev-thm} also equals the number of positive eigenvalues of $\mathcal{L}$.

%	If producing a computer assisted proof is not of concern, then our methodology can be easily adapted to compute and count the conjugate points with standard numerics. 
%By Lemma \ref{lem:conj-det} one would only need to compute a frame for $\mathbb{E}^\mathrm{u}_-(x)$ and then look for values of $x$ where $\mathrm{det}A_1(x) = 0$.  

To produce a computer assisted proof several things need to be made rigorous.  
At the most basic level, it is generally the case that even the stationary solution $\varphi$  cannot be computed exactly. 
Beyond that,  one must quantify the initial approximation error of the frame matrix, how this error grows when numerically integrating \eqref{E:eval-sys}, and how large $-L_- $ and $L_+$ must be to ensure our count of conjugate points is complete.

%We seek to answer those questions in this work.

%
%\begin{Remark}
%	If producing a computer assisted proof is not of concern, then our methodology can be easily adapted to compute and count the conjugate points with standard numerics. 
%	By Lemma \ref{lem:conj-det} one would only need to compute a frame for $\mathbb{E}^\mathrm{u}_-(x)$ and then look for values of $x$ where $\mathrm{det}A_1(x) = 0$. 
%	To compute this frame matrix, first one defines an approximate frame matrix for $\mathbb{E}^\mathrm{u}_-(-L_-)$ by taking a $2n \times n$ matrix $A(-L_-)$ whose columns are the $n$-unstable eigenvectors of $J \mathcal{B}_-(0)$. 
%	Then for $-L_- \leq  x \leq L_+$, the frame matrix $A(x)= \left( \begin{smallmatrix} A_1(x) \\ A_2(x) \end{smallmatrix} \right)$ is defined column by column by numerically integrating the differential equation in \eqref{E:eval-sys}. 
%	
%	To produce a computer assisted proof several things need to be made rigorous, such as quantifying the initial approximation error of the frame matrix, how this error grows when numerically integrating \eqref{E:eval-sys}, and how large must $-L_- $ and $L_+$ be to ensure our count of conjugate points is complete. We seek to answer those questions in this work.
%	
%	{\color{blue} I've modified this remark. Double check for wording / placement - JJ}
%\end{Remark}

To that end we use validated numerics to rigorously compute  \emph{a posteriori} error bounds.
Thus, we are able to quantify the error produced in our numerical algorithms ranging from the imprecision in solving initial value problems, down to the rounding error from finite precision arithmetic.  
Early landmark results of validated numerics include the computer assisted proofs of the universality of the Feigenbaum constants \cite{lanford1982computer}  and  Smale's 14\textsuperscript{th} problem  on the nature of the Lorenz attractor \cite{tucker2002rigorous}. 
Validated numerics, also referred to as  rigorous numerics, have found great success in providing computer assisted proofs of non-perturbative results in dynamics, both finite and infinite dimensional, and we refer the interested reader to \cite{van2015rigorous,gomez2019computer,nakao2019numerical} for further details.  

As mentioned earlier, the foundation of our numerical method begins with computing the stationary solution $ \varphi$, a homoclinic/heteroclinic solution to the spatial dynamical system association with stationary solutions of \eqref{E:gradrd}.  
The computation of connecting orbits has been, and continues to be, an active area of research within the validated numerics community, see for example  \cite{oishi1998numerical,wilczak2003heteroclinic,van2015stationary,ArioliKoch15,capinski2017beyond,van2018continuation,van2020validated} and the references cited therein. 
In Section \ref{sec:Methodology} we describe the approach taken in computing and proving the existence of  $\varphi$ and, in addition, the validated computation of its conjugate points and computer assisted proof of its spectral (in)stability.

We note that, although our method is developed for equations of the form \eqref{E:eval}, we expect that it could be applied more broadly. In particular, there are equations that are not systems of reaction-diffusion equations with gradient structure whose associated eigenvalue problems possess symplectic structure similar to that of \eqref{E:eval-sys}. See \cite{ChardardDiasBridges09,ChardardDiasBridges11, Howard20} for a variety of examples, including connections between those eigenvalue problems and the Maslov index. 

Moreover, as mentioned above, the relationship between unstable eigenvalues and conjugate points has also been proven in several contexts with many space dimensions \cite{CoxJonesLatushkin16,CoxJonesMarzuola15, DengJones11}. Recent work has shown that the corresponding path of Lagrangian subspaces, analogous to $\mathbb{E}^u_-(x; 0)$ here, can be understood via an ill-posed dynamical system \cite{BeckCoxJones20, BeckCoxJones21}. Validated numerics have been applied in other settings involving ill-posed dynamical systems \cite{CastelliGameiroLessard18,arioli2020traveling}, and so it is possible that the methods of this paper could be utilized to study multidimensional stability in the future.

The remainder of the paper is organized as follows. In \S\ref{S:Ls} we explicitly characterize the values of $L_\pm$.  In \S\ref{sec:Methodology} we describe the associated rigorous computations for determining stability by first computing the underlying wave itself and then computing the associated conjugate points. In \S \ref{S:example1} we apply our results to an example problem, which shows that fronts in reaction-diffusion systems of the form \eqref{E:gradrd} need not be stable, as they are in the scalar case. Although not surprising, this can be viewed as complementary to the result in \cite[\S 5]{BeckCoxJones18} that shows that symmetric pulses in such systems must be unstable, as in the scalar case. Moreover, it provides a proof-of-concept for this work. Finally, in \S\ref{S:future-directions} we discuss future directions for our work.

%%%%%%%%%%%%%%%%%%%%%%%%%%%%%%%%%%%%%%%%%%%%%%%%%%%%%%%%%%%%

\section{Determination of $L_\pm$}\label{S:Ls}

We begin with a first order eigenvalue problem of the form \eqref{E:eval-sys} that results from a reaction-diffusion system of the form \eqref{E:gradrd}. Due to Theorem \ref{thm:prev-thm}, to count unstable eigenvalues we only need to compute conjugate points for $\lambda = 0$. Therefore, we rewrite \eqref{E:eval-sys} for $\lambda = 0$ as
\begin{equation}\label{E:main-sys}
U_x = \tilde{\mathcal{A}}(x)U, \qquad \tilde{\mathcal{A}}(x) = \begin{pmatrix} 0 & D^{-1} \\ -\mathcal{M}(x) & 0 \end{pmatrix}.
\end{equation}
We need to determine values $L_\pm$ such that all conjugate points are contained in the interval $[-L_-, L_+]$. Recall that conjugate points will be determined by the evolution of the unstable subspace, $\mathbb{E}^u_-(x; \lambda = 0) =: \mathbb{E}^u_-(x)$. The main idea used to determine $L_\pm$ is that, for $|x|$ sufficiently large, the evolution of $\mathbb{E}^u_-(x)$ is essentially governed by the dynamics determined by the asymptotic matrices $\lim_{x\to\pm\infty} \tilde{\mathcal{A}}(x)$. We begin with $L_-$
in \S\ref{S:L_-}. The determination of $L_+$, which is somewhat more involved due to the presence of the eigenfunction $\varphi'(x)$, is presented in \S\ref{S:L_+}.

%%%%%%%%%%%%%%%%%%%%%%%%%%%%%%%%

\subsection{Determination of $L_-$}\label{S:L_-} 

To determine $L_-$, using (H3) we rewrite \eqref{E:main-sys} as
\begin{equation}\label{E:eval-sys-minus-infinity}
U_x = [\mathcal{A}_{-\infty} + \mathcal{A}_-(x)]U, \qquad \mathcal{A}_{-\infty} = \begin{pmatrix} 0 & D^{-1} \\ -\mathcal{M}_- & 0 \end{pmatrix}, \qquad  \mathcal{A}_-(x) = \begin{pmatrix} 0 & 0 \\ -\mathcal{M}(x) + \mathcal{M}_- & 0 \end{pmatrix}, 
\end{equation}
where
\begin{equation}\label{E:A-bound}
\|\mathcal{A}_-(x)\| \leq C_M e^{-\eta_M |x|}, \qquad x \leq 0.
\end{equation}
The goal of this section is to prove Proposition \ref{prop:L-}, which gives an explicit way to choose $L_-$ so there are no conjugate points for $x \leq -L_-$. We begin with a preliminary result.

\begin{Lemma}\label{lem:no-asym-conj} The matrix $\mathcal{A}_{-\infty}$ is hyperbolic, its eigenvalues are real and distinct, and its eigenvectors form a basis for $\mathbb{R}^{2n}$. Furthermore, if $\mathbb{E}^u_{-\infty}$ is the subspace spanned by its eigenvectors corresponding to positive eigenvalues, then $\mathbb{E}^u_{-\infty} \cap \mathcal{D} = \{0\}$.
\end{Lemma}

\begin{Proof} First, we relate the eigenvalues and eigenvectors of $\mathcal{A}_{-\infty}$
%, which we denote by $\{\nu_j^-\}$ and $\{V_j^-\}$ respectively, 
to $\mathcal{M}_-$. In particular, if $(u,w)$ is an eigenvector of $\mathcal{A}_{-\infty}$ with eigenvalue $\nu$, then
\[
\nu \begin{pmatrix} u \\ w \end{pmatrix} = \mathcal{A}_{-\infty} \begin{pmatrix} u \\ w \end{pmatrix}  = \begin{pmatrix} 0 & D^{-1} \\ -\mathcal{M}_- & 0 \end{pmatrix}\begin{pmatrix} u \\ w \end{pmatrix}.
\]
This implies that 
\[
w = \nu D u, \qquad D^{-1} \mathcal{M}_- u = -\nu^2 u.
\]
By Hypotheses (H1)-(H2), $\mathcal{M}_-$ is symmetric and negative. If we define $\gamma = -\nu^2$ and $z = D^{1/2}u$, then $D^{-1/2}\mathcal{M}_- D^{-1/2} z = \gamma z$. The matrix $D^{-1/2}\mathcal{M}_- D^{-1/2}$ is symmetric because $\mathcal{M}_-$ is symmetric and $D$ is diagonal, and it is also negative:
\[
\langle D^{-1/2}\mathcal{M}_- D^{-1/2} u, u \rangle = \langle \mathcal{M}_- D^{-1/2} u, D^{-1/2} u \rangle < 0,
\]
because $\mathcal{M}_-$ is negative. Thus, $\gamma = -\nu^2 < 0$. Moreover, we can choose the $n$ eigenvectors of $D^{-1/2}\mathcal{M}_- D^{-1/2}$, which we denote by $\{z_j^-\}_{j=1}^n$, to be an orthonormal basis for $\mathbb{R}^n$. We denote the corresponding $n$ (negative) eigenvalues via $\{\gamma_j^-\}_{j=1}^n$. Hence, the eigenvalues of $\mathcal{A}_{-\infty}$ with positive real part, which are in fact just positive, and their corresponding eigenvectors are given by
\begin{equation}\label{E:asym-evecs}
V_j^{-, u} = \begin{pmatrix} D^{-1/2} z_j^- \\ \nu_j^- D^{1/2} z_j^- \end{pmatrix}, \qquad \nu_j^{-} = + \sqrt{-\gamma_j^-}, \qquad j=1, \dots, n.
\end{equation}
Similarly, the negative eigenvalues and their corresponding eigenvectors are given by
\[
V_j^{-, s} = \begin{pmatrix} D^{-1/2} z_j^- \\ -\nu_j^- D^{1/2} z_j^- \end{pmatrix}, \qquad j=1, \dots, n.
\]
By hypothesis (H2) the $2n$ eigenvalues $\{\nu_j^{-}, -\nu_j^{-}\}_{j=1}^n$ of $\mathcal{A}_{-\infty}$ are all distinct. Furthermore, the fact that $ \det ( D^{1/2}) \neq 0 $ and  $\{z_j^-\}_{j=1}^n$ is a basis for $\mathbb{R}^n$ implies that $\{D^{-1/2}z_j^-\}_{j=1}^n$ is also a basis for $\mathbb{R}^n$, although it need not be orthonormal. The fact that $\{D^{-1/2}z_j^-\}_{j=1}^n$ is a basis for $\mathbb{R}^n$ implies $\mathbb{E}^u_{-\infty} \cap \mathcal{D} = \{0\}$.
\end{Proof}

We now construct the solutions $\{\tilde V_j^{-,u}(x)\}_{j=1}^n$ of \eqref{E:eval-sys-minus-infinity} that are asymptotic to $\{V_j^{-, u}\}_{j=1}^n$. Given a solution $V(x)$ of \eqref{E:eval-sys-minus-infinity}, define $W(x)$ via $V(x) = e^{\nu x}W(x)$, where $\nu$ is an eigenvalue of $\mathcal{A}_{-\infty}$ with eigenvector $V^-$. We then have
\begin{equation}\label{E:shifted-ode}
W_x = [\mathcal{A}_{-\infty} - \nu]W + \mathcal{A}_-(x) W. 
\end{equation}
To construct solutions to this equation, we will utilize the exponential dichotomy determined by the constant matrix $\mathcal{A}_{-\infty} - \nu I$. To that end, we now review some well-known facts about dichotomies associated with constant matrices.

Given any constant matrix $A \in \mathbb{R}^{N \times N}$ with eigenvalues $\{\nu_j\}_{j=1}^N$ ordered so that $\nu_j > \nu_{j+1}$, define the diagonal matrix  $\Lambda = \mathrm{diag}\{ \nu_1 , \dots , \nu_N\}$ and define $Q \in \mathbb{R}^{N \times N}$ to be a matrix whose columns are the eigenvectors of $A$. Hence $ A = Q \Lambda Q^{-1}$ and $ e^{A x } = Q e^{\Lambda x} Q^{-1}$. 
For fixed $m$, define $N\times N$ matrices $\Lambda_{u} = \mathrm{diag} \{\nu_1 , \dots \nu_{m-1} , 0 , \dots 0\}$ and $\Lambda_{cs} = \mathrm{diag} \{0 , \dots 0, \nu_{m} , \dots \nu_N \}$. If $P^{cs}$ is the projection onto the eigenspace with eigenvalues $\nu \leq \nu_m$ and $P^u$ is the complementary projection onto the eigenspace with eigenvalues $\nu > \nu_m$, then 
\begin{align*}
e^{(A - \nu_m)x} P^{cs} &= Q e^{ (\Lambda_{cs} - \nu_m) x} Q^{-1} P^{cs}
& e^{(A - \nu_m)x} P^{u} &= Q e^{ (\Lambda_u - \nu_m) x} Q^{-1} P^u.
\end{align*}
As $ \| P^u \|, \| P^{cs} \| =1$ with respect to the Euclidean norm, if we define
\begin{align} \label{eq:K_-def}
K_- &= \| Q \| \cdot \|Q^{-1} \|
\end{align}
then we obtain the bound
\[
\|e^{(A-\nu_m)x} P^{cs} \| \leq K_-\mbox{ for } x \geq 0, \qquad \|e^{(A-\nu_m)x} P^{u}\| \leq K_- e^{(\nu_{m-1}-\nu_m) x} \mbox{ for } x \leq 0.
\] 
Note that if the matrix of eigenvectors $Q$ can be chosen to be orthonormal, then $K_- = 1$. 

\begin{Remark}
It is not necessary that the eigenvalues satisfy the strict inequality $\nu_j > \nu_{j+1}$. If there are repeated eigenvalues with independent eigenvectors, the above construction is still valid. This could be relevant if, for example, (H2) is relaxed to allow for repeated eigenvalues. If there are Jordan blocks, then some modification is necessary to obtain an essentially equivalent result, but this case is not relevant as long as $\mathcal{M}_\pm$ are symmetric. 
\end{Remark}

Returning to equation \eqref{E:shifted-ode}, let $P^{cs}$ be the projection onto the eigenspace of $\mathcal{A}_{-\infty} - \nu$ corresponding to eigenvalues with real part less than or equal to zero, and let $P^u$ be the projection onto the eigenspace of $\mathcal{A}_{-\infty} - \nu$ corresponding to eigenvalues with real part greater than zero. Consider the map
\begin{eqnarray}
W(x) &=& V^- + \int_{-\infty}^x e^{(\mathcal{A}_{-\infty} - \nu)(x-y)}P^{cs}\mathcal{A}_-(y)W(y) \rmd y \nonumber \\
&& \qquad \qquad - \int_x^{-L}e^{(\mathcal{A}_{-\infty} - \nu)(x-y)}P^{u}\mathcal{A}_-(y)W(y) \rmd y \nonumber \\
&=:& \mathcal{T}_-(W)(x), \label{E:def-T}
\end{eqnarray}
where $L_-$ is any fixed positive constant. Note that, because $(\mathcal{A}_{-\infty} - \nu)V^- = 0$, any fixed point of this map must be a solution of \eqref{E:shifted-ode}. Let $K_-$ and $\eta_-$ be positive constants such that 
\begin{eqnarray}
\|e^{(\mathcal{A}_{-\infty} - \nu)(x-y)}P^{cs}\| &\leq& K_-, \quad y \leq x \leq 0, \nonumber \\ 
\|e^{(\mathcal{A}_{-\infty} - \nu)(x-y)}P^{u}\| &\leq& K_- e^{-\eta_-(y-x)}, \quad x \leq y \leq 0. \label{E:dichotomy-estimates}
\end{eqnarray}

\begin{Proposition}\label{prop:L-contraction}
Fix $ L_- >0$, define $ K_- $ such that \eqref{E:dichotomy-estimates} holds, and define 
\begin{align*}
\lambda_{\mathcal{T}_-} :=  \frac{K_-C_M}{\eta_M}e^{-\eta_M L_-},
\end{align*}
where $C_M$ and $\eta_M$ are the constants appearing in Hypothesis (H3). If $\lambda_{\mathcal{T}_-} < 1$, then the map $\mathcal{T}_-$ defined in \eqref{E:def-T} is a contraction on $X = L^\infty((-\infty, -L_-], \R^{2n} )$  with contraction constant $\lambda_{\mathcal{T}_-} $. Its unique fixed point $W(x)$ satisfies
\begin{equation}\label{E:minus-infty-asymptotics}
\|W(\cdot) - V^-\|_X \leq \frac{\lambda_{\mathcal{T}_-}}{1-\lambda_{\mathcal{T}_-}} |V^-|.
\end{equation}
\end{Proposition}

\begin{Proof}
We will use properties of the dictotomy defined by the projection operators $P^{cs}$ and $P^u$. We note that similar calculations have appeared elsewhere; see, eg, \cite[\S 4]{BeckSandstedeZumbrun10} and \cite[\S 2]{HowardLatushkinSukhtayev18}. Equation \eqref{E:dichotomy-estimates} together with \eqref{E:A-bound} implies
\begin{align*}
|\mathcal{T}_-(W_1)(x) - \mathcal{T}_-(W_2)(x)| & \leq \int_{-\infty}^x K_- C_M e^{-\eta_M|y|} \rmd y \|W_1-W_2\|_X   \\
& \qquad+  \int_x^{-L_-}K_- e^{-\eta_-(y-x)} C_M e^{-\eta_M|y|} \rmd y \|W_1- W_2\|_X  \\
&\leq \int_{-\infty}^{-L_-} K_-C_M  e^{\eta_M y} \rmd y \|W_1- W_2\|_X \\
&\leq \frac{K_-C_M}{\eta_M}e^{-\eta_M L_-} \|W_1- W_2\|_X.
\end{align*}
As a result, for $L_-$ sufficiently large, this map is a contraction with contraction constant 
\begin{equation}\label{E:contraction-constant}
\lambda_{\mathcal{T}_-} :=  \frac{K_-C_M}{\eta_M}e^{-\eta_M L_-}  < 1.
\end{equation}
We denote its unique fixed point by $W(x)$, and we wish to estimate $|W(x) - V^-|$. To do this, first apply $\mathcal{T}_-$ to the constant vector $V^-$ and note that by essentially the same estimate as above,
\[
\|\mathcal{T}_-(V^-)(\cdot) - V^- \|_X \leq \lambda_{\mathcal{T}_-} |V^-|.
\]
Next, due to the telescoping series, we have
\[
\mathcal{T}_-^n (V^-) - V^- = \sum_{j=0}^{n-1} [\mathcal{T}^{j+1}_- (V^-) - \mathcal{T}^j_- (V^-)] \qquad \Rightarrow \qquad W(x) - V^- = \sum_{j=0}^{\infty} [\mathcal{T}^{j+1}_-(V^-) - \mathcal{T}^j_- (V^-)].
\]
As a result, 
\begin{equation}\label{E:eigendirection-estimate}
\|W(\cdot) - V^- \|_X \leq \sum_{j=0}^{\infty} \lambda_{\mathcal{T}_-}^j \|\mathcal{T}_-(V^-) - V^-\|_X  = \frac{\lambda_{\mathcal{T}_-}|V^-|}{1-\lambda_{\mathcal{T}_-}}.
\end{equation}
\end{Proof}

\begin{Remark}
Only the constant $K_-$ in \eqref{E:dichotomy-estimates} appears in the final constant $\lambda_{\mathcal{T}_-}$, given in \eqref{E:contraction-constant}. The value of $\eta_-$ is not needed. We included it in the statement of \eqref{E:dichotomy-estimates} for completeness. 
\end{Remark}

We can apply Proposition \ref{prop:L-contraction} for $\nu = \nu_j^{-}$, $j = 1, \dots, n$, to construct the desired solutions $\{\tilde V_j^{-,u}(x) = e^{\nu_j^{-} x} W_j^{-,u}(x)\}$ of \eqref{E:eval-sys-minus-infinity} that are asymptotic, as $x \to -\infty$, to $\{e^{\nu_j^{-} x} V_j^{-,u}\}_{j=1}^n$. (We used tildes here to distinguish between the constant vectors $V_j^{-,u}$ and the $x$-dependent solutions $\tilde V_j^{-,u}(x)$.) Notice that
\[
\mathbb{E}^u_-(x) = \mathrm{span}\{\tilde V_1^{-,u}(x), \dots, \tilde V_n^{-,u}(x) \} = \mathrm{span}\{W_1^{-,u}(x), \dots, W_n^{-,u}(x)  \}.
\]
Since $W_j^{-,u}(x)$ is asymptotically close to $V^{-,u}_j$ in the sense of \eqref{E:minus-infty-asymptotics}, this implies that for $L_-$ large $\mathbb{E}^u_-(-L_-)$ should be close to $\mathrm{span}\{ V_1^{-,u}, \dots, V_n^{-,u}\} = \mathbb{E}^u_{-\infty}$. Therefore, since $\mathbb{E}^u_{-\infty} \cap \mathcal{D} = \{0\}$ (see Lemma \ref{lem:no-asym-conj}), this should imply that $\mathbb{E}^u_{-}(-L_-) \cap \mathcal{D} = \{0\}$, as well.

\begin{Proposition} \label{prop:L-}
Recall that $D = \mathrm{diag}(d_1, \dots, d_n) > 0$. Let $d_{max} = \max_{j}\{d_j\}$ and $d_{min} = \min_{j}\{d_j\}$. If $L_-$ is chosen so that
\begin{equation}\label{E:L-condition}
\frac{\lambda_{\mathcal{T}_-}}{1-\lambda_{\mathcal{T}_-}}  < \frac{d_{min}^{1/2}}{d_{max}^{1/2}(1 + \|\mathcal{M}_-\| d_{max})^{1/2}\sqrt{n}},
\end{equation}
where $\lambda_{\mathcal{T}_-}$ is defined in \eqref{E:contraction-constant}, then $\mathbb{E}^u_-(x) \cap \mathcal{D} = \{0\}$ for all $x \leq -L_-$. 
\end{Proposition}

\begin{Proof}
Recall that 
\[
\mathbb{E}^u_-(x) = \mathrm{span}\{W_1^{-,u}(x), \dots, W_n^{-,u}(x)  \}, \qquad \|W_j^{-,u}(\cdot) - V^{-,u}_j \|_X \leq \frac{\lambda_{\mathcal{T}_-}|V_j^{-,u}|}{1-\lambda_{\mathcal{T}_-}}.
\]
For each $j$ we can therefore write 
\begin{equation}\label{E:utilde-est}
W_j^{-,u}(x) = \begin{pmatrix} u_j(x) \\ w_j(x) \end{pmatrix}, \qquad u_j(x) = u_j^- + \tilde u_j(x), \qquad \|\tilde u_j(\cdot) \|_X \leq \frac{\lambda_{\mathcal{T}_-} |V_j^{-,u}|}{1-\lambda_{\mathcal{T}_-}}, 
\end{equation}
where $u_j^- = D^{-1/2}z_j^-$ as given in \eqref{E:asym-evecs}. Suppose that $x_0 \in (-\infty, -L_-]$ is a conjugate point, ie $\mathbb{E}^u_-(x_0) \cap \mathcal{D} \neq \{0\}$. Then there exists $\{c_j\}_{j=1}^n$ so that
\begin{equation} \label{E:conj-pt-cond}
0 = \sum_{j=1}^n c_j u_j(x_0) = \sum_{j=1}^n c_j [u_j^- + \tilde u_j(x_0)] \quad \Rightarrow \quad  \sum_{j=1}^n c_j u_j^- = - \sum_{j=1}^n c_j \tilde u_j(x_0).
\end{equation}
Since $u_j^- = D^{-1/2}z_j^-$ and $\{z_j^-\}_{j=1}^n$ is an orthonormal basis for $\mathbb{R}^n$ (see Lemma  \ref{lem:no-asym-conj}),
\[
\left| \sum_{j=1}^n c_j u_j^- \right| = \left| D^{-1/2} \sum_{j=1}^n c_j z_j^-\right| \geq \frac{1}{d_{max}^{1/2}} \left|\sum_{j=1}^n c_j z_j^-\right| = \frac{1}{d_{max}^{1/2}} \sqrt{c_1^2 + \dots c_n^2}.
\]
On the other hand, by \eqref{E:utilde-est},
\[
\left| - \sum_{j=1}^n c_j \tilde u_j(x_0) \right| \leq \frac{\lambda_{\mathcal{T}_-}}{1-\lambda_{\mathcal{T}_-}} \sum_{j=1}^n |c_j| |V_j^{-,u}|.
\]
Moreover, again using \eqref{E:asym-evecs} and the proof of Lemma \ref{lem:no-asym-conj},
\begin{eqnarray*}
|V_j^{-,u}|^2 &=& |u_j^-|^2 + |w_j^-|^2 = |D^{-1/2}z_j^-|^2 + |\nu_j^{-} D^{1/2}z_j^-|^2 \leq \frac{1}{d_{min}} + \max_{j} |\nu_j^-|^2 d_{max} \\
&\leq&  \frac{1}{d_{min}}(1 + \|\mathcal{M}_-\| d_{max}),
\end{eqnarray*}
where we have used the fact that
\[
-(\nu_j^-)^2 = \langle D^{-1/2}\mathcal{M}_- D^{-1/2} z_j^-, z_j^- \rangle = \langle \mathcal{M}_- D^{-1/2} z_j^-, D^{-1/2}z_j^- \rangle,
\]
which implies
\[
|(\nu_j^-)^2| \leq \|\mathcal{M}_-\| \frac{1}{d_{min}}.
\]
Combining these results, we find \eqref{E:conj-pt-cond} cannot hold if
\[
\frac{1}{d_{max}^{1/2}} \sqrt{c_1^2 + \dots c_n^2} > \frac{\lambda_{\mathcal{T}_-}}{1-\lambda_{\mathcal{T}_-}}\frac{1}{d_{min}^{1/2}}(1 + \|\mathcal{M}_-\| d_{max})^{1/2} \sum_{j=1}^n |c_j|. 
\]
Using the fact that 
\[
\left( \sum_{j=1}^n c_j^2 \right)^{1/2} \geq \frac{1}{\sqrt{n}} \sum_{j=1}^n |c_j|
\]
we obtain the result. 
\end{Proof}

%%%%%%%%%%%%%%%%%%%%%%%%%%%%%%%%

\subsection{Determination of $L_+$}\label{S:L_+}

To determine $L_+$, using (H3) we rewrite \eqref{E:main-sys} as
\begin{equation}\label{E:eval-sys-plus-infinity}
U_x = [\mathcal{A}_{+\infty} + \mathcal{A}_+(x)]U, \qquad \mathcal{A}_{+\infty} = \begin{pmatrix} 0 & D^{-1} \\ -\mathcal{M}_+ & 0 \end{pmatrix}, \qquad  \mathcal{A}_+(x) = \begin{pmatrix} 0 & 0 \\ -\mathcal{M}(x) + \mathcal{M}_+ & 0 \end{pmatrix}, 
\end{equation}
where
\begin{equation}\label{E:A-bound-2}
\|\mathcal{A}_+(x)\| \leq C_M e^{-\eta_M |x|}, \qquad x \geq 0.
\end{equation}
The goal of this section is to prove the following:
\begin{Theorem}\label{thm:main}
For any $x \geq L_+$, where $L_+$ is sufficiently large so that the right hand side of \eqref{eq:ProjectionMatrixBound} is strictly less than one, $\mathbb{E}^u_-(x) \cap \mathcal{D} = \{0\}$.
\end{Theorem}

By hypothesis (H4), we know the that $\mathbb{E}^u_-(x) \cap \mathbb{E}^s_+(x) = \mathrm{span}\{  (\varphi', D\varphi'') \}$, where $\mathbb{E}^s_+(x) = \mathbb{E}^s_+(x; \lambda=0)$ is the subspace of solutions to \eqref{E:main-sys} that decay to zero as $x \to +\infty$. Denote $U_\varphi = (\varphi', D\varphi'')$ and write
\begin{equation}\label{E:subspace-rep}
\mathbb{E}^u_-(x) = \mathrm{span}\{ U_\varphi(x), U_1(x), \dots, U_{n-1}(x)\}, \qquad \mathbb{E}^u_-(x) \in \Lambda(n) \quad \forall x,
\end{equation}
where the solutions $U_k(x)$ satisfy $U_k(x) \notin \mathbb{E}^s_+(x)$ for all $k$. In order to determine $L_+$, we would like to be use the fact that, for large $x$, $\mathbb{E}^u_-(x)$ will look like the asymptotic unstable eigenvectors at $+\infty$, with the exception of the asymptotic behavior of $U_\varphi(x)$.

\begin{figure}[h]
\centering 
\includegraphics[width=0.9\linewidth]{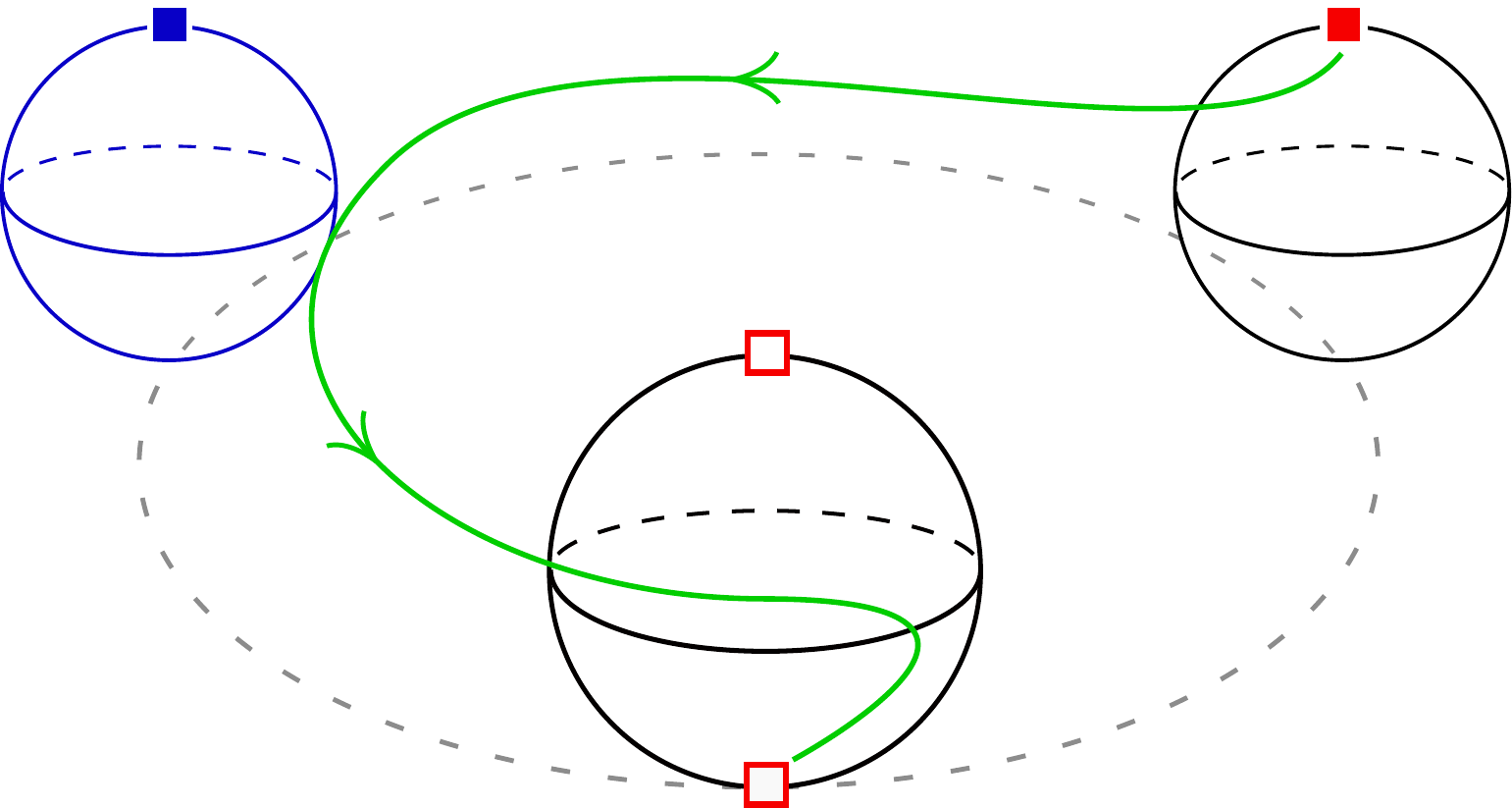}
\caption{  A cartoon  of $ \Lambda(2)$, which is double covered by $ S^1 \times S^2$. The upper left sphere represents the set of planes which nontrivially intersect $\cD$ (in fact homotopic to a Klein bottle with one loop collapsed), and $ \cD$ is  represented by the upper left square. The trajectory $\mathbb{E}^u_-(x)$  is represented by the green line.}
\label{fig:L(2)}
\end{figure}
One way to think of the evolution of $\mathbb{E}^u_-(x)$ is as follows. In Figure \ref{fig:L(2)}, we have depicted a cartoon of the path $\mathbb{E}^u_-(x)$ traces out in $ \Lambda(n)$. 
The beginning of this trajectory $\mathbb{E}_-^u(- \infty)$ is the unique point $\mathbb{E}_{-\infty}^u \in \Lambda(n)$, depicted as solid square in the upper right sphere. 
By counting the number of conjugate points of $\mathbb{E}^u_-(x)$ (ie the number of nontrivial intersections with $\cD$ with orientation) one is able compute the Maslov index of the path $\mathbb{E}^u_-(x)$, which for 
a loop $\gamma: S^1 \to \Lambda(n)$ is given by its equivalence class in the fundamental group $[\gamma] \in \pi_1(\Lambda(n)) \cong \Z$.

It follows from Hypothesis \ref{H:potential} and \ref{H:simple} that  $\lim_{x \to +\infty} U_{\varphi}(x) / | U_{\varphi}(x)| $  must limit to one of the $n$ stable eigendirections of $ \cA_+$. Furthermore, $\mathbb{E}_-^u(x)$ must limit as $ x \to + \infty$ to a Lagrangian plane in $\Lambda(n)$ spanned by $\lim_{x \to +\infty} U_{\varphi}(x) / | U_{\varphi}(x)| $ and $n-1$ unstable eigenvectors of $ \cA_+$. 
Consequently  $\mathbb{E}_-^u(+ \infty)$ must be one of $n$ possible points in $\Lambda(n)$, depicted as the two hollow squares in the bottom sphere of Figure \ref{fig:L(2)}. Note that each of these $n$ possible limit points in $\Lambda(n)$ is unstable under the flow defined by \eqref{E:eval-sys-plus-infinity}. Hence showing that $\mathbb{E}_-^u(x)$ limits to a particular one of these points can be difficult. However such a calculation is not necessary; to compute $L_+$, we do not need to show that $\mathbb{E}_-^u(L_+)$  is close to its limit point, but rather it suffices to show that $\mathbb{E}_-^u(x)$  is bounded away from $\cD$ for $ x \geq L_+$. 

To begin our analysis, we wish to obtain a good asymptotic description of the solutions that are asymptotic to the eigendirections at $+\infty$, just as we did in \S \ref{S:L_-} at $-\infty$. In particular, we can construct solutions to \eqref{E:eval-sys-plus-infinity} of the form
\[
\tilde V_j^{+,s}(x) = e^{-\nu_j^+ x} (V_j^{+,s} + \tilde{W}_j^{+, s}(x)), \qquad \tilde V_j^{+,u}(x) = e^{\nu_j^{+} x} (V_j^{+,u} + \tilde{W}_j^{+,u}(x)), \qquad j = 1, \dots n,
\]
where
\begin{equation}\label{E:asym-evectors}
V_j^{+,s} = \frac{1}{\sqrt{2\nu_j^+}}\begin{pmatrix} D^{-1/2}z_j^+ \\ -\nu_j^+ D^{1/2}z_j^+ \end{pmatrix}, \qquad V_j^{+,u} =  \frac{1}{\sqrt{2\nu_j^+}}\begin{pmatrix} D^{-1/2}z_j^+ \\ \nu_j^+ D^{1/2}z_j^+ \end{pmatrix}.
\end{equation}
Using notation analogous to the previous section, we denote $W_j^{+, s/u}(x) = V_j^{+,s/u} + \tilde{W}_j^{+, s/u}(x)$. 

The vectors $\{z_j^+\}_{j=1}^n$ are eigenvectors of the negative symmetric matrix $D^{-1/2}\mathcal{M}_+D^{-1/2}$ and form an orthonormal basis for $\mathbb{R}^n$. The corresponding eigenvalues are $\gamma_j^+ = -(\nu_j^+)^2$. We will assume the eigenvalues, which are distinct due to (H2), are ordered so that
\[
\nu_1^+ > \nu_2^+ > \dots > \nu_n^+  > 0.
\]
Note that we have normalized the eigenvectors $\{V_j^{+, s/u}\}$ so that they form a symplectic basis. (This normalization was not necessary in \S\ref{S:L_-}.) In particular,
\begin{equation}\label{E:asym-lag}
\omega(V_j^{+, s/u}, V_k^{+, s/u}) = 0 \quad \mbox{if} \quad  k \neq j, \qquad \omega(V_j^{+,u}, V_j^{+,s}) = 1.
\end{equation}
Analogous to Proposition \ref{prop:L-contraction} we have the following result, which will lead to bounds on the error terms $\tilde{W}_j^{+, s/u}$.
\begin{Proposition}\label{prop:L+bounds} Let
\begin{equation}\label{E:contraction-constant-plus}
\lambda_{\mathcal{T}_+} :=  \frac{K_+C_M}{\eta_M}e^{-\eta_M L_+},
\end{equation}
where the constants $C_M$ and $\eta_M$ are such that $\|\mathcal{A}_+(x)\| \leq C_M e^{\eta_M x}$ for $x \geq 0$ as guaranteed by $(H3)$, and the constant $K_+$ is defined similarly to \eqref{E:dichotomy-estimates} but using $\mathcal{A}_{+\infty}$. For any $L_+ > 0$ such that $\lambda_{\mathcal{T}_+} < 1$, the functions $W_j^{+, s/u}(x)$ satisfy
\begin{equation}\label{E:plus-infty-asymptotics} 
\|W^{+, s/u}_j(\cdot) - V^{+, s/u}_j\|_{L^\infty([L_+, \infty),\R^{2n})} \leq \frac{\lambda_{\mathcal{T}_+}}{1 - \lambda_{\mathcal{T}_+}} |V_j^{+, s/u}|.
\end{equation}
\end{Proposition}
Define
\begin{equation}\label{E:def-epsilon0}
\epsilon_0 := \frac{\lambda_{\mathcal{T}_+}}{1 - \lambda_{\mathcal{T}_+}}  \max_{j,s,u} |V_j^{+, s/u}|,
\end{equation}
which we can make small by choosing $L_+$ large. The Proposition implies 
\begin{equation}\label{E:W-error}
|\tilde{W}_j^{+, s/u}(x)| \leq \epsilon_0, \qquad \forall \quad x \geq L_+, \qquad \forall j, s, u.
\end{equation}

The functions $\{\tilde V_j^{+, s}(x)\}_{j=1}^n$ form a basis for $\mathbb{E}^s_+(x)$ which is unique up to multiplication by an element of $GL(n)$. 
The functions $\{\tilde V_j^{+, u}(x)\}_{j=1}^n$ form one basis for $\mathbb{E}^u_+(x)$. If we let $\tilde V^{+, u}(x)$ and $\tilde V^{+, s}(x)$ denote the matrices whose columns are $\tilde V_j^{+, s/u}(x)$ then, in general, a basis for $\mathbb{E}^u_+(x)$ may be given by the columns of $  \tilde V^{+, u}(x) \cdot M_1 + \tilde V^{+, s}(x) \cdot  M_2$ for arbitrary matrices $M_1 \in GL(n)$ and $M_2 \in \R^{n \times n}$. 

Recall \eqref{E:subspace-rep}. Using the functions $\{\tilde V_j^{+, s/u}(x)\}_{j=1}^n$ as a basis for arbitrary solutions on $x \geq 0$, we can write
\begin{eqnarray}
U_\varphi(x) &=&  \sum_{j=1}^n \tilde \alpha_j e^{-\nu_j^+ x} (V_j^{+,s} + \tilde{W}_j^{+,s}(x)), \label{E:soln-exp} \\
U_k(x) &=& \sum_{j=1}^n[ \tilde \gamma_j^k e^{\nu_j^+ x} (V_j^{+,u} + \tilde{W}_j^{+,u}(x)) + \tilde \beta_j^k e^{-\nu_j^+ x} (V_j^{+,s} + \tilde{W}_j^{+,s}(x))], \nonumber
\end{eqnarray}
for all $x \geq 0$, where $\tilde \alpha_j, \tilde \beta_j^k, \tilde \gamma_j^k \in \mathbb{R}$ are constants that are uniquely determined. For notational convenience below, we will for fixed $x$ write this as
\begin{equation}\label{E:soln-exp-2}
U_\varphi(x) =  \sum_{j=1}^n  \alpha_j (V_j^{+,s} + \tilde{W}_j^{+,s}(x)), \qquad 
U_k(x) = \sum_{j=1}^n[ \gamma_j^k  (V_j^{+,u} + \tilde{W}_j^{+,u}(x)) +  \beta_j^k (V_j^{+,s} + \tilde{W}_j^{+,s}(x))],
\end{equation}
where $\alpha_j = \tilde \alpha_j e^{-\nu_j^+ x}$, $\beta_j^k = \tilde \beta_j^k e^{-\nu_j^+ x}$, and $\gamma^k_j = \tilde \gamma_j^k e^{\nu_j^+ x}$. 
We also note that these solutions can be written as
\begin{eqnarray}
U_\varphi(x) &=& (\mathcal{V}^s + \mathcal{W}^s(x))\alpha = (\mathcal{V}^s + \mathcal{W}^s(x))\Lambda^{-1} \tilde \alpha \label{E:soln-exp-3} \\
U_k(x) &=& (\mathcal{V}^s + \mathcal{W}^s(x))\beta^k + (\mathcal{V}^u + \mathcal{W}^u(x))\gamma^k = (\mathcal{V}^s + \mathcal{W}^s(x))\Lambda^{-1} \tilde \beta^k + (\mathcal{V}^u + \mathcal{W}^u(x))\Lambda \tilde \gamma^k, \nonumber
\end{eqnarray}
where
\[
\mathcal{V}^{s,u} = [V_1^{+, s/u} | \dots | V_n^{+, s/u}] \in \mathbb{R}^{2n \times n}, \qquad \mathcal{W}^{s,u}(x) = [\tilde{W}_1^{+, s/u}(x) | \dots | \tilde{W}_n^{+, s/u}(x)] \in \mathbb{R}^{2n \times n}, 
\]
\[
\Lambda(x) = \mathrm{diag}(e^{\nu_1^+ x}, \dots e^{\nu_n^+ x}) \in \mathbb{R}^{n \times n},
\]
and for each $ x \geq L_+$ and $ k$ we have 
\begin{align*}
\alpha &= (\alpha_1, \dots, \alpha_n) \in \mathbb{R}^{n}, 
	&
\gamma^k &= (\gamma_1^k, \dots, \gamma^k_n) \in \mathbb{R}^{n}, 
	 &
\beta^k &= (\beta_1^k, \dots, \beta^k_n) \in \mathbb{R}^{n},
\end{align*}
with corresponding definitions for $\tilde \alpha$, $\tilde \beta^k$, and $\tilde \gamma^k$. Note that, because $U_k(x) \notin \mathbb{E}^s_+(x)$, it must be true that $\gamma^k \neq 0$ for all $k$. Moreover, the $\gamma^k$'s must be independent. Otherwise, if $\gamma^k = c \gamma^j$, then $U_k - cU_j \in \mathbb{E}^s_+(x)$, which would contradict Hypothesis (H4), that the zero eigenvalue is simple. 

\begin{Remark}
The vectors $\tilde \alpha, \tilde \beta^k, \tilde \gamma^k$ can be computed numerically with known error bounds. 
\end{Remark}

The Lagrangian condition says
\[
0 = \omega(U_\varphi(x), U_k(x)) = \omega((\mathcal{V}^s + \mathcal{W}^s(x))\alpha, (\mathcal{V}^u + \mathcal{W}^u(x)) \gamma^k) \qquad \forall k,
\]
where we have used the fact that $\mathbb{E}^s_+(x) \in \Lambda(n)$ implies $\omega((\mathcal{V}^s + \mathcal{W}^s(x))\alpha, (\mathcal{V}^s + \mathcal{W}^s(x))\beta^k) = 0$. Let's write
\[
\mathcal{V}^s + \mathcal{W}^s(x) = \begin{pmatrix} \mathcal{V}^s_1 + \mathcal{W}^s_1(x) \\ \mathcal{V}^s_2 + \mathcal{W}^s_2(x) \end{pmatrix}, \qquad \mathcal{V}^u + \mathcal{W}^u(x) = \begin{pmatrix} \mathcal{V}^u_1 + \mathcal{W}^u_1(x) \\ \mathcal{V}^u_2 + \mathcal{W}^u_2(x) \end{pmatrix}.
\]
Using the explicit formulas for $\mathcal{V}^{+, s/u}_{1,2}$ given by \eqref{E:asym-evectors}, we find
\begin{equation}\label{E:explicitVs}
\mathcal{V}^s_1 = \mathcal{V}^u_1 = D^{-1/2}\mathcal{Z}\mathcal{N}, \quad -\mathcal{V}^s_2 = \mathcal{V}^u_2 = \frac{1}{2}D^{1/2}\mathcal{Z}\mathcal{N}^{-1}, \quad \mathcal{N} = \mathrm{diag}\left(\frac{1}{\sqrt{2\nu_j^+}}\right),
\end{equation}
and 
\[
\mathcal{Z} = [z_1^+ | \dots | z_n^+], \qquad \mathcal{Z}^T \mathcal{Z} = I.
\]
Using this notation, we can write the above Lagrangian condition as
\begin{align*}
0 &= \omega((\mathcal{V}^s + \mathcal{W}^s(x))\alpha, (\mathcal{V}^u + \mathcal{W}^u(x))\gamma^k) \\
&=  -\langle (\mathcal{V}^s_1 + \mathcal{W}^s_1(x)) \alpha, (\mathcal{V}^u_2 + \mathcal{W}^u_2(x))\gamma^k \rangle 
+ \langle (\mathcal{V}^s_2 + \mathcal{W}^s_2(x)) \alpha, (\mathcal{V}^u_1 + \mathcal{W}^u_1(x))\gamma^k \rangle \\
%%%%%%
&= -\langle \alpha, 
[
 (\mathcal{V}^s_1 + \mathcal{W}^s_1(x))^T (\mathcal{V}^u_2 + \mathcal{W}^u_2(x))
 -
 (\mathcal{V}^s_2 + \mathcal{W}^s_2(x))^T(\mathcal{V}^u_1 + \mathcal{W}^u_1(x)) 
 ]\gamma^k \rangle \\
%%%%%%%
&=  -\langle \alpha, [
1 
- (\mathcal{V}_2^s)^T\mathcal{W}^u_1(x) 
+ (\mathcal{V}_1^s)^T\mathcal{W}^u_2(x) 
- (\mathcal{W}_2^s(x))^T(\mathcal{V}^u_1 + \mathcal{W}^u_1(x)) 
+ (\mathcal{W}_1^s(x))^T(\mathcal{V}^u_2 + \mathcal{W}^u_2(x))
]\gamma^k \rangle.
\end{align*}
Therefore, we have found
\begin{eqnarray}
\alpha &\perp& \Gamma_1, \qquad \Gamma_1 = \mathrm{span}\{ \Gamma^k_1: k = 1, \dots, n-1\} \label{E:first-condition} \\
 \Gamma^k_1 &=& \gamma^k - [(\mathcal{V}_2^s)^T\mathcal{W}^u_1(x) - (\mathcal{V}_1^s)^T\mathcal{W}^u_2(x) + (\mathcal{W}_2^s(x))^T(\mathcal{V}^u_1 + \mathcal{W}^u_1(x)) - (\mathcal{W}_1^s(x))^T(\mathcal{V}^u_2 + \mathcal{W}^u_2(x))]\gamma^k. \nonumber
\end{eqnarray}
This means that a frame matrix for the $n-1$ dimensional subspace $\Gamma_1\subset \mathbb{R}^n$ is given by
\begin{eqnarray}
\Gamma_1: && \quad \Gamma + Q\Gamma, \qquad \Gamma = [\gamma^1 | \dots | \gamma^{n-1}] \label{E:Gamma1-frame} \\
 Q &=& (\mathcal{V}_1^s)^T\mathcal{W}^u_2(x) - (\mathcal{V}_2^s)^T\mathcal{W}^u_1(x) + (\mathcal{W}_1^s(x))^T(\mathcal{V}^u_2 + \mathcal{W}^u_2(x)) - (\mathcal{W}_2^s(x))^T(\mathcal{V}^u_1 + \mathcal{W}^u_1(x)). \nonumber
\end{eqnarray}

Suppose now that there is an intersection of $\mathbb{E}^u_-(x)$ with the Dirichlet subspace. Since this involves the first $n$ components of vectors in $\mathbb{R}^{2n}$, we introduce the projection operator $\pi_1: \mathbb{R}^{2n} \to \mathbb{R}^{2n}$, defined by
\begin{align} \label{eq:ProjTop}
	\pi_1(u_1, \dots, u_{2n}) &= (u_1, \dots u_n, 0, \dots 0) \in \mathbb{R}^{2n}.
\end{align}
With this notation, an intersection of $\mathbb{E}^u_-(x)$ with the Dirichlet subspace means 
\begin{align*}
0 &= \det [ \pi_1 U_{\varphi'}(x) , \pi_1 U_{1}(x) , \dots , \pi_1 U_{n-1}(x)  ] \\
0 &= \mathrm{det}[ (\mathcal{V}^s_1 + \mathcal{W}^s_1(x))\alpha, (\mathcal{V}^u_1 + \mathcal{W}^u_1(x))\gamma^1 + (\mathcal{V}^s_1 + \mathcal{W}^s_1(x))\beta^1, \dots, (\mathcal{V}^u_1 + \mathcal{W}^u_1(x))\gamma^{n-1} + (\mathcal{V}^s_1 + \mathcal{W}^s_1(x))\beta^{n-1}] \\
%%%%%%%%%%%
&= \mathrm{det}(\mathcal{V}^s_1 + \mathcal{W}^s_1(x))\\
%%%%%%%%%%%%%%%%%%
&\qquad  \times \mathrm{det}[\alpha, \beta^1 + (\mathcal{V}^s_1 + \mathcal{W}^s_1(x))^{-1}(\mathcal{V}^u_1 + \mathcal{W}^u_1(x))\gamma^1, \dots, \beta^{n-1} + (\mathcal{V}^s_1 + \mathcal{W}^s_1(x))^{-1}(\mathcal{V}^u_1 + \mathcal{W}^u_1(x))\gamma^{n-1}].
\end{align*}
Moreover,
\begin{eqnarray*}
(\mathcal{V}^s_1 + \mathcal{W}^s_1(x))^{-1}(\mathcal{V}^u_1 + \mathcal{W}^u_1(x)) &=& (D^{-1/2}\mathcal{Z}\mathcal{N} + \mathcal{W}^s_1(x))^{-1}(D^{-1/2}\mathcal{Z}\mathcal{N} + \mathcal{W}^u_1(x)) \\
&=& [I + (D^{-1/2}\mathcal{Z}\mathcal{N})^{-1} \mathcal{W}^s_1(x)]^{-1} [I + (D^{-1/2}\mathcal{Z}\mathcal{N})^{-1} \mathcal{W}^u_1(x)].
\end{eqnarray*}
Thus,
\begin{eqnarray}
\alpha &\in& \Gamma_2, \qquad \Gamma_2 = \mathrm{span}\{ \beta^k + \Gamma_2^k: k = 1, \dots, n-1\} \label{E:second-condition} \\
\Gamma_2^k &:=& [I + \mathcal{N}^{-1} \mathcal{Z}^T D^{1/2} \mathcal{W}^s_1(x)]^{-1} [I + \mathcal{N}^{-1} \mathcal{Z}^T D^{1/2} \mathcal{W}^u_1(x)]\gamma^k. \nonumber
\end{eqnarray}
This means a frame matrix for $\Gamma_2$ is given by
\begin{eqnarray}
\Gamma_2: && \quad \Gamma + P\Gamma + B, \qquad \Gamma = [\gamma^1 | \dots | \gamma^{n-1}], \qquad B = [\beta^1| \cdots | \beta^{n-1}] \label{E:Gamma2-frame} \\
 P &=& \mathcal{N}^{-1} \mathcal{Z}^T D^{1/2} \mathcal{W}^u_1(x) + \left( \sum_{k=1}^\infty (-1)^k (\mathcal{N}^{-1} \mathcal{Z}^T D^{1/2} \mathcal{W}^s_1(x))^k \right)[I + \mathcal{N}^{-1} \mathcal{Z}^T D^{1/2} \mathcal{W}^u_1(x)], \nonumber
\end{eqnarray}
where, via \eqref{E:W-error}, we are assuming that $L_+$ is sufficiently large so that 
\begin{equation}\label{E:extra-L+-cond}
\|\mathcal{N}^{-1} \mathcal{Z}^T D^{1/2} \mathcal{W}^s_1(x)\| < 1
\end{equation} 
for all $x \geq L_+$. We note this will follow from the condition that $\epsilon_0 \sqrt{2n \nu_1^+ d_{max}} < 1$, which appears in the statement of Proposition \ref{prop:main2}. Let $\pi_{\Gamma_1}, \pi_{\Gamma_2} \in \mathbb{R}^{n\times n}$ be the projections onto $\Gamma_{1,2}$, respectively.

\begin{Proposition}\label{prop:main}
There exists an $L_+$ sufficiently large so that, for any $x \geq L_+$, $\|\pi_{\Gamma_1} - \pi_{\Gamma_2}\| < 1$.  
\end{Proposition} 
We are interested in being able to verify for a specific $L_+$ whether it is sufficiently large so that $\|\pi_{\Gamma_1} - \pi_{\Gamma_2}\| < 1$ for all $ x \geq L_+$.  
Hence we find it necessary to explicitly quantify Proposition \ref{prop:main} in full detail with the following proposition. Below we will prove Proposition \ref{prop:main2}, which implies also Proposition \ref{prop:main}. Define 
\begin{equation} \label{E:BGtilde}
\tilde B = [\tilde \beta^1 | \dots | \tilde \beta^{n-1}] , \qquad \tilde \Gamma = [\tilde \gamma^1 | \dots | \tilde \gamma^{n-1}].
\end{equation}
\begin{Proposition}\label{prop:main2} Fix $L_+$ and $ \epsilon_0(L_+) >0$. 
For $	D := \mathrm{diag}(d_1, \dots, d_n)$ define the following constants: 
	\begin{align*} 
	%%%%%%%%%%%%%%%%% 
	\epsilon_b &:=\frac{1}{\sqrt{\mu^*}} e^{-2 \nu_n^+ L_+ } \| \tilde{B} \| 
	&
	\mu^* &:= \min\{ \mu: \mu \in \sigma(\tilde \Gamma^T \tilde \Gamma)\}
	\\
	d_{min} &:= \mathrm{min}\{d_j: j = 1, \dots, n\}, 
	&
	 d_{max} &:= \mathrm{max}\{d_j: j = 1, \dots, n\} \\
	 %%%%%%
 C_P &:=	 \frac{2   \sqrt{2 n\nu_1^+ d_{max}}}{1 - \epsilon_0  \sqrt{2n \nu_1^+ d_{max}}} 
 &
 C_Q  &:=    \sqrt{\frac{2n}{\nu_n^+ d_{min}}} + \sqrt{2n\nu_1^+ d_{max}} + 2 \epsilon_0 n  
\\
 C_{M_1} &:=  \epsilon_0 C_Q ( 2 + \epsilon_0 C_Q) 
&
C_{M_2} &:= \epsilon_0 C_P ( 2+ \epsilon_0 C_P)  + 2 \epsilon_b ( 1 + \epsilon_0 C_P) + \epsilon_b^2 .
	\end{align*}
If $C_{M_1} ,C_{M_2} ,\epsilon_0  \sqrt{2n\nu_1^+ d_{max}}<1$, then 	 for all $ x \geq L_+$ we have: 
\begin{align}
\| \pi_{\Gamma_1} - \pi_{\Gamma_2} \| &< 2 \epsilon_0 ( C_Q + C_P) + \epsilon_0^2 ( C_Q^2 + C_P^2) \nonumber \\
&\qquad 
+ \epsilon_b^2 + 
2 \epsilon_b ( 1 + \epsilon_0 C_P)  \label{eq:ProjectionMatrixBound} \\
&\qquad 
+ (1 + \epsilon_0 C_Q )^2 \frac{C_{M_1}}{1 - C_{M_1}}
+ (1 + \epsilon_0 C_P + \epsilon_b )^2 \frac{C_{M_2}}{1 - C_{M_2}} . \nonumber 
\end{align}
\end{Proposition}
Note that $ C_{M_1}$ and $C_{M_2}$ contain factors of $\epsilon_0$ and $ \epsilon_b$, and therefore every term on the right hand side of \eqref{eq:ProjectionMatrixBound} contains at least one factor of either $\epsilon_0$ or $ \epsilon_b$, both of which go to $0$ as $ L_+ \to +\infty$.

\begin{Remark}\label{rem:gamma}
A key issue in estimating $\|\pi_{\Gamma_1} - \pi_{\Gamma_2}\|$ will be controlling the exponentially growing factors present in any terms involving $\Gamma$. One might naively think that $\|\Gamma\| \sim e^{\nu_1^+ x}$, and so $\|(\Gamma^T\Gamma)^{-1} \| \sim e^{-2\nu_1^+ x}$. However, this is not necessarily the case. To see this, consider the example for $n=3$ given by
\[
\Gamma = \begin{pmatrix} e^{\nu_1^+ x} & e^{\nu_1^+ x} \\ 0 & e^{\nu_2^+ x} \\ 0 & 0 \end{pmatrix} \sim e^{\nu_1^+ x}.
\]
One can explicitly compute
\[
(\Gamma^T \Gamma)^{-1} = \frac{1}{ e^{2\nu_1^+ x+ 2\nu_2^+ x}} \begin{pmatrix} e^{2\nu_1^+ x} +  e^{2\nu_2^+ x} & -e^{2\nu_1^+ x} \\ - e^{2\nu_1^+ x} & e^{2\nu_1^+ x} +  e^{2\nu_2^+ x} \end{pmatrix} \sim e^{-2\nu_2^+ x}
\]
Therefore, a naive estimate would predict $\Gamma(\Gamma^T \Gamma)^{-1} \Gamma^T  \sim e^{\nu_1^+ x}e^{-2\nu_2^+ x}e^{\nu_1^+ x}$, which grows exponentially in $x$ and hence is insufficient. However, a direct calculation shows that
\[
\Gamma(\Gamma^T \Gamma)^{-1} \Gamma^T = \begin{pmatrix} 1 & 0 & 0 \\ 0 & 1 & 0 \\ 0 & 0 & 0 \end{pmatrix}
\]
This makes sense, since $\pi_\Gamma = \Gamma(\Gamma^T \Gamma)^{-1} \Gamma^T$ is a projection matrix, so its norm should not grow (or decay) exponentially in $x$. In fact, $\|\pi_\Gamma\| = 1$, as is the case for any (nontrivial) projection. This means, however, that in estimating $\|\pi_{\Gamma_1} - \pi_{\Gamma_2}\|$, we must not lose too much information about the structure of $\Gamma_{1,2}$.  
\end{Remark}

\begin{Proof}{\bf (of Theorem \ref{thm:main})}
Suppose that there exists an $x \geq L_+$ such that $\mathbb{E}^u_-(x) \cap \mathcal{D} \neq \{0\}$. Then \eqref{E:first-condition},  \eqref{E:second-condition}, and the above propositions imply
\[
|\alpha| = |\pi_{\Gamma_2}\alpha| = |\pi_{\Gamma_2} \alpha - \pi_{\Gamma_1} \alpha| < |\alpha|,
\]
which is a contradiction. Therefore, we obtain Theorem \ref{thm:main} as an immediate corollary.
\end{Proof}

\begin{Proof}{\bf (of Propositions \ref{prop:main} and \ref{prop:main2})}
Recall that, if $V$ is a frame matrix for a given $k$-dimensional subspace of $\mathbb{R}^n$, then
\[
\pi_V = V(V^TV)^{-1}V^T
\] 
is the corresponding projection matrix. The fact that $\mathrm{rank} \mbox{ }V = k$ ensures that $V^TV$ is invertible. 

The idea in bounding $ \|\pi_{\Gamma_1} - \pi_{\Gamma_2}\|$ will be to combine as many $\Gamma$'s as possible into terms of the form $\Gamma(\Gamma^T \Gamma)^{-1} \Gamma^T$, which is a projection and hence has unit norm. We begin by collecting some bounds that will be used in estimating each term. 

Using this formula and equations \eqref{E:Gamma1-frame} and \eqref{E:Gamma2-frame}, a direct calculation shows that
\begin{align*}
	\pi_{\Gamma_1} &= 
	\Gamma ( \Gamma^T \Gamma)^{-1} \Gamma + 
	Q \Gamma ( \Gamma^T \Gamma)^{-1} \Gamma_1^T \\
   & \qquad + 
	\Gamma ( \Gamma^T \Gamma)^{-1} (Q\Gamma)^T + 
	\Gamma_1 \left( \sum_{k=1}^\infty (-1)^k (M_1)^k \right) (\Gamma ^T \Gamma)^{-1}  \Gamma_1^T \\
	%%%%%%%%%%%%%%
	\pi_{\Gamma_2} &= 
	\Gamma ( \Gamma^T \Gamma)^{-1} \Gamma + 
	(P \Gamma +B)( \Gamma^T \Gamma)^{-1} \Gamma_2^T \\
	& \qquad + 
	\Gamma ( \Gamma^T \Gamma)^{-1} (P\Gamma + B)^T + 
	\Gamma_2 \left( \sum_{k=1}^\infty (-1)^k (M_2)^k \right) (\Gamma ^T \Gamma)^{-1}  \Gamma_2^T 
	%%%%%%%%%%%%%%
\end{align*}
where we have defined
\begin{eqnarray}
M_1 &=& (\Gamma^T\Gamma)^{-1}[\Gamma^TQ\Gamma + (Q\Gamma)^T\Gamma + (Q\Gamma)^T(Q\Gamma)]   \label{E:def-Ms}\\ 
M_2 &=& (\Gamma^T\Gamma)^{-1}[\Gamma^T(P\Gamma + B) + (P\Gamma+B)^T\Gamma + (P\Gamma+B)^T(P\Gamma+B)] 
\nonumber
\end{eqnarray}
so that 
\begin{align*}
\Gamma_1^T \Gamma_1  & = (	\Gamma^T \Gamma) ( I  + M_1) , 
&
\Gamma_2^T \Gamma_2  & = (	\Gamma^T \Gamma) ( I  + M_2).  
\end{align*}
By taking a difference and gathering terms, we obtain the following:  
\begin{align}
\pi_{\Gamma_1} - \pi_{\Gamma_2}  &= \underbrace{(Q-P)  \pi_{\Gamma} + Q \pi_{\Gamma} Q^T + P \pi_{\Gamma} P^T + \pi_{\Gamma} (Q-P)^T}_{I} \nonumber \\
& \qquad 	- \underbrace{B (\Gamma^T \Gamma)^{-1}  B^T - B (\Gamma^T \Gamma)^{-1} \Gamma^T ( I + P)^T - ( I + P) \Gamma (\Gamma^T \Gamma)^{-1}  B^T}_{II} \nonumber \\
& \qquad + \underbrace{\Gamma_1 \left( \sum_{k=1}^\infty (-1)^k (M_1)^k \right) (\Gamma ^T \Gamma)^{-1}  \Gamma_1^T - \Gamma_2 \left( \sum_{k=1}^\infty (-1)^k (M_2)^k \right) (\Gamma ^T \Gamma)^{-1}  \Gamma_2^T}_{III}  \label{eq:ProjectionDifference}
\end{align}

We bound the norm of the pieces I, II, and III of \eqref{eq:ProjectionDifference} in turn. 

\begin{Remark}\label{rem:scaling-factor}
	The $\gamma$'s and $\beta$'s are not really unique; any scalar multiple of them also works. Our estimate should take this into account. In other words, consider the basis solution $\tilde U_k = c_k U_k$, which would be equivalent to instead taking $c_k \gamma^k$ and $c_k \beta^k$. If we define
	\[
	C = \mathrm{diag}(c_1, \dots, c_{n-1}) \in \mathbb{R}^{(n-1)\times (n-1)}, 
	\]
	this would be equivalent to replacing $\Gamma$ and $B$ with $\Gamma C$ and $B C$, respectively. Inspecting \eqref{E:def-Ms} and \eqref{eq:ProjectionDifference}, each term involves $(\Gamma^T\Gamma)^{-1}$ and either two factors of $\Gamma$ or one factor of $\Gamma$ and one factor of $B$. Thus, it seems plausible the constant matrix $C$ would not have an overall effect on the estimate. This will be confirmed below.
\end{Remark}

%%%%%%%  I  %%%%%%%
\textbf{Bounds for I:}   \newline 
We first show that if $\epsilon_0  \sqrt{2n\nu_1^+ d_{max}} < 1$ then:
\begin{align}
\| Q \| & \leq \epsilon_0 C_Q,
&
\| P \| & \leq \epsilon_0 C_P .
\end{align}
Given any matrix $M$, one has
\begin{equation}\label{E:norm-characterization}
\|M\| = \max_{\lambda \in \sigma(M^TM)}\sqrt{|\lambda|}
\end{equation}
To estimate $P$ and $Q$, first note that \eqref{E:W-error} implies that $\epsilon_0$ is an upper bound on every component of the $ n \times n$ matrices $\|\mathcal{W}^{s,u}_{1,2}(x)^T \mathcal{W}^{s,u}_{1,2}(x)\|$, and thereby  
\begin{equation}\label{E:Wsu-bound}
\|\mathcal{W}^{s,u}_{1,2}(x)\| \leq  \epsilon_0 \sqrt{n}.
\end{equation}
In addition, \eqref{E:explicitVs} and the fact that $\|\mathcal{Z}\| = 1$ leads to
\begin{equation}\label{E:Vbounds}
\|\mathcal{V}^{s,u}_1\| \leq \frac{1}{\sqrt{2\nu_n^+ d_{min}}}, \qquad \|\mathcal{V}^{s,u}_2\| \leq \sqrt{\frac{\nu_1^+ d_{max}}{2}}. 
\end{equation}

Using the formula for $Q$ in \eqref{E:Gamma1-frame}, as well as \eqref{E:Wsu-bound}, and \eqref{E:Vbounds}, we find
\begin{equation}\label{E:Qbound}
\|Q\| \leq  \epsilon_0  \left( \sqrt{\frac{2 n }{\nu_n^+ d_{min}}} + \sqrt{2 n \nu_1^+ d_{max}} + 2 \epsilon_0 n\right) = \epsilon_0 C_Q.
\end{equation}
Using the formula for $P$ in \eqref{E:Gamma2-frame}, and again \eqref{E:Wsu-bound}, and \eqref{E:Vbounds}, we find 
\begin{equation}\label{E:Pbound}
\|P\| \leq \frac{2\epsilon_0  \sqrt{2n\nu_1^+ d_{max}}}{1 - \epsilon_0  \sqrt{2n\nu_1^+ d_{max}}} = \epsilon_0 C_P,
\end{equation}
where in the above we have assumed that $\epsilon_0$ is sufficiently small so that $\epsilon_0  \sqrt{2n\nu_1^+ d_{max}} < 1$.

Hence we obtain:
\begin{align}
\| 	(Q-P)  \pi_{\Gamma} + Q \pi_{\Gamma} Q^T + P \pi_{\Gamma} P^T + \pi_{\Gamma} (Q-P)^T \| & \leq 2 \epsilon_0 ( C_Q + C_P) + \epsilon_0^2 ( C_Q^2 + C_P^2) \label{eq:MainPart1}
\end{align}

%%%%%%%  II  %%%%%%%
\textbf{Bounds for II:}  \newline 
We first show that for $\epsilon_b$ defined as in our proposition, we have:
\begin{align*}
\|\Gamma(\Gamma^T\Gamma)^{-1}B^T\| &\leq \epsilon_b ,
&
\|B(\Gamma^T\Gamma)^{-1}B^T\| &\leq \epsilon_b^2 .
\end{align*}
Consider the term $\Gamma(\Gamma^T\Gamma)^{-1}B^T$. 
Notice that
\begin{equation}\label{E:sym-matrix}
[\Gamma(\Gamma^T\Gamma)^{-1}B^T]^T\Gamma(\Gamma^T\Gamma)^{-1}B^T = B (\Gamma^T\Gamma)^{-1} B^T.
\end{equation}
Furthermore, using the fact that $\Gamma = \Lambda \tilde \Gamma$ and $B = \Lambda^{-1}\tilde B$, we therefore need to understand the norm of
\begin{equation}\label{E:see-invariance}
B (\Gamma^T\Gamma)^{-1} B^T = \Lambda^{-1} \tilde B (\tilde \Gamma^T \Lambda^2 \tilde \Gamma)^{-1} \tilde B^T \Lambda^{-1}.
\end{equation}

We first focus on the matrix $\tilde \Gamma^T \Lambda^2 \tilde \Gamma$, which is real and symmetric. Hence, if $\{\lambda_j\}$ are its eigenvalues, then 
$\| (\tilde \Gamma^T \Lambda^2 \tilde \Gamma)^{-1}\| \leq \lambda_{min}^{-1}$. If $u$ is a unit eigenvector with eigenvalue $\lambda$, then 
\[
\lambda = \langle u, \lambda u\rangle = \langle u, \tilde \Gamma^T \Lambda^2 \tilde \Gamma u\rangle= \langle \tilde \Gamma u,  \Lambda^2 \tilde \Gamma u\rangle \geq e^{2\nu_n^+ x} \langle \tilde \Gamma u, \tilde \Gamma u\rangle  = e^{2\nu_n^+ x} \langle  u, \tilde \Gamma ^T \tilde \Gamma u\rangle.
\]
The matrix $\tilde \Gamma ^T \tilde \Gamma$ is again real and symmetric (and of full rank). Moreover, its eigenvalues are positive, because if $\mu$ is an eigenvalue with unit eigenvector $q$, then $\mu = \langle q, \tilde \Gamma ^T \tilde \Gamma q \rangle = |\tilde \Gamma q|^2$. Hence, 
\[
\lambda_{min} \geq e^{2\nu_n^+ x} \min\{ \mu: \mu \in \sigma(\tilde \Gamma^T \tilde \Gamma)\}.
\]
Denoting
\[
\mu^* = \min\{ \mu: \mu \in \sigma(\tilde \Gamma^T \tilde \Gamma)\},
\]
we find
\[
\| (\tilde \Gamma^T \Lambda^2 \tilde \Gamma)^{-1}\| \leq \frac{1}{\mu^*}e^{-2\nu_n^+ x}.
\]
Since $\|\Lambda^{-1}\| \leq e^{-\nu_n^+ x}$, we therefore find
\[
\|B(\Gamma^T\Gamma)^{-1}B^T\| \leq \frac{1}{\mu^*}e^{-4\nu_n^+ x} \|\tilde B\|^2 =\epsilon_b^2.
\]
Since $B(\Gamma^T\Gamma)^{-1}B^T$ is real and symmetric, this bound also provides a bound on its eigenvalues. As a result, 
\eqref{E:norm-characterization} and \eqref{E:sym-matrix} imply
\begin{equation}\label{E:second-bound}
\|\Gamma(\Gamma^T\Gamma)^{-1}B^T\|  \leq \frac{1}{\sqrt{\mu^*}}e^{-2\nu_n^+ L_+} \|\tilde B\| = \epsilon_b, \qquad \mu^* = \min\{ \mu: \mu \in \sigma(\tilde \Gamma^T \tilde \Gamma)\}.
\end{equation}

Hence we can obtain the estimate:
\begin{align}
\| 
 B (\Gamma^T \Gamma)^{-1}  B^T
+ B (\Gamma^T \Gamma)^{-1} \Gamma^T ( I + P)^T
+ ( I + P) \Gamma (\Gamma^T \Gamma)^{-1}  B^T \| & \leq  \epsilon_b^2 + 
2 \epsilon_b ( 1 + \epsilon_0 C_P)  \label{eq:MainPart2}
\end{align}

\begin{Remark}
	Recall Remark \ref{rem:scaling-factor}. If in \eqref{E:see-invariance} we replace $\tilde \Gamma$ and $\tilde B$ by $\tilde \Gamma C$ and $\tilde B C$, respectively, for some invertible diagonal matrix $C$, then we find
	\[
	\Lambda^{-1} \tilde B C (C \tilde \Gamma^T \Lambda^2 \tilde \Gamma C)^{-1} \tilde C B^T \Lambda^{-1} =  \Lambda^{-1} \tilde B (\tilde \Gamma^T \Lambda^2 \tilde \Gamma)^{-1} \tilde B^T \Lambda^{-1},
	\]
	and so the estimate does not depend on the matrix $C$. 
\end{Remark}

\begin{Remark}
	The quantity $\mu^* = \min\{ \mu: \mu \in \sigma(\tilde \Gamma^T \tilde \Gamma)\}$ appearing in \eqref{E:second-bound} will in practice be computed numerically. Although we know for theoretical reasons that $\mu^* > 0$, one must be careful that any (controllable) numerical error does not cause a violation of this condition. 
\end{Remark}

%%%%%%%  III  %%%%%%%
\textbf{Bounds for III:} \newline 
Keeping Remark \ref{rem:gamma} in mind, we simplify the expressions for $M_{1,2}$ and III as follows. First, \eqref{E:def-Ms} can be written as:
\begin{align}
	M_1 &= (\Gamma^T\Gamma)^{-1}\Gamma^T  \tilde M_1 \Gamma  
	&
	M_2 &= (\Gamma^T\Gamma)^{-1}\Gamma^T \tilde M_2 \Gamma  \label{E:def-Ms-better} 
\end{align}
where we define $\tilde M_1 := Q + Q^T + Q^T Q  $ and 
\begin{align*}
\tilde M_2 &:= P + P^T + P^TP \\&\qquad + (1+P^T)B(\Gamma^T\Gamma)^{-1}\Gamma^T + \Gamma(\Gamma^T\Gamma)^{-1}B^T (1+P) + \Gamma (\Gamma^T\Gamma)^{-1}B^TB(\Gamma^T\Gamma)^{-1}\Gamma^T. \nonumber
\end{align*}
A straightforward calculation shows that 
$C_{M_1} \geq \|\tilde M_1\| $ and $C_{M_2} \geq \|\tilde M_2\| $. 
 Using \eqref{E:def-Ms-better}, notice that
\[
M_1^k (\Gamma^T \Gamma)^{-1}\Gamma^T = (\Gamma^T \Gamma)^{-1}\Gamma^T[\tilde{M}_1 \pi_\Gamma]^k.
\]
This can be shown, for example, via an induction argument.
As a result, 
\[
\Gamma \left( \sum_{k=1}^\infty (-1)^kM_1^k \right) (\Gamma^T \Gamma)^{-1} \Gamma^T = \pi_\Gamma \sum_{k=1}^\infty (-1)^k [\tilde{M}_1 \pi_\Gamma]^k,
\]
and so
\begin{equation*}%\label{E:third-bound-part1}
\left\|\Gamma \left( \sum_{k=1}^\infty (-1)^kM_1^k \right) (\Gamma^T \Gamma)^{-1} \Gamma^T \right\| 
\leq 
\frac{C_{M_1}}{1-C_{M_1}}. 
\end{equation*}
Similarly, we obtain the following estimate 
\begin{equation*}%\label{E:third-bound-part2}
\left\|\Gamma \left( \sum_{k=1}^\infty (-1)^kM_2^k \right) (\Gamma^T \Gamma)^{-1} \Gamma^T \right\| \leq \frac{ C_{M_2}}{1-C_{M_2}}.
\end{equation*}
By factoring out $ \Gamma_1 = (I + Q) \Gamma$, we can go about bounding our infinite sums as follows:
\begin{align}
\big\|	\Gamma_1 \left( \sum_{k=1}^\infty (-1)^k (M_1)^k \right) (\Gamma ^T \Gamma)^{-1}  \Gamma_1^T \big\|
& =  (1+\epsilon_0 C_Q)^2 \frac{C_{M_1}}{1-C_{M_1}}. \label{eq:MainPart3a}
\end{align}
Similarly, by writing $\Gamma_2 (\Gamma^T\Gamma)^{-1}  = 
( I + P +  B (\Gamma^T \Gamma)^{-1} \Gamma^T)  \Gamma (\Gamma^T\Gamma)^{-1} $  we obtain: 
\begin{align}
\big\|	\Gamma_2 \left( \sum_{k=1}^\infty (-1)^k (M_2)^k \right) (\Gamma ^T \Gamma)^{-1}  \Gamma_2^T \big\|
& =  (1+\epsilon_0 C_P + \epsilon_b)^2 \frac{C_{M_2}}{1-C_{M_2}}. \label{eq:MainPart3b}
\end{align}

Thus, by plugging into \eqref{eq:ProjectionDifference} the estimates from \eqref{eq:MainPart1}, \eqref{eq:MainPart2}, \eqref{eq:MainPart3a} and \eqref{eq:MainPart3b}, we obtain \eqref{eq:ProjectionMatrixBound}. 
By choosing $L_+$ large we can control the size of $\epsilon_0$ and $e^{-2\nu_n^+ x}$, and hence also $C_{M_2}$, which means we can choose $L_+$ so the right hand side of \eqref{eq:ProjectionDifference} is strictly less than $1$. This completes the proof.
\end{Proof}

%%%%%%%%%%%%%%%%%%%%%%%%%%%%%%%%%%%%%%%%%%%%%%%%%%%%%%%%%%%%%%%%%%%
%\input{Application_Section}

\section{Methodology for Computing Conjugate Points}
\label{sec:Methodology}

Here we demonstrate how to rigorously count conjugate points in practice. Again, we restrict our interest to PDEs of the form in \eqref{E:gradrd}, repeated below: 
\begin{equation} \label{eq:GradPDE}
v_t = Dv_{xx} + \nabla G(v), \qquad G \in C^2(\mathbb{R}^n, \mathbb{R}).
\end{equation} 
Moreover, we are interested in computing the spectral stability of an equilibrium solution $ \varphi: \R\to \R^n$ such that its linearization \eqref{E:eval} satisfies Hypotheses \ref{H:potential} and \ref{H:simple}. To do so, we perform the following steps: 
\begin{enumerate}
\item  For two points $v_0 , v_1 \in \R^{n}$ such that $ 0=  \nabla G(v_0)  = \nabla G(v_1)$ and $G(v_0) =  G(v_1)$, compute (and prove the existence of) a standing wave $\varphi : \R\to \R^n$  such that  $ \lim_{x \to - \infty} \varphi(x) = v_0$ and $\lim_{x \to +\infty}\varphi(x) = v_1$. 
\item Fix $L_-\geq 0$ and prove that $\bE_{-}^{u}(x) \cap \cD = \{0\}$ for all $ x \leq -L_-$.  
\item Fix $L_+\geq0$ and calculate a frame matrix $ U:[-L_-,L_+] \to \R^{2n\times n}$ whose columns span $ \bE_-^u(x)$ for all $ x \in [-L_-,L_+]$.  
\item Prove that  $\bE_{-}^{u}(x) \cap \cD = \{0\}$ for all $ x \geq L_+$.   
\item Count the conjugate points on $ [ - L_- , L_+]$. 
\end{enumerate}
By Theorem \ref{thm:prev-thm} the number of unstable eigenvalues of $ \mathcal{L}$ in \eqref{E:eval} is equal to the number of conjugate points $ x \in \R$  where $ \bE_{-}^u(x) \cap \cD \neq \{0\}$. By Lemma \ref{lem:conj-det},  $x \in \R$ is a conjugate point if and only if $ \det A_1(x) =0$. Using this approach, we reduce this search for conjugate points down to a finite interval $ [-L_- , L_+]$. We note  that in practice we will fix $L_+ =0$ and leave $L_-$ free to be chosen. We note also that our computational methodology may also be adapted for use with standard/non-validated numerics. For this case, in short, one may skip steps (ii) and (iv).

Being a numerical algorithm several computational parameters, such as $L_-$, must be fixed. To obtain a more theoretical flavor, one could rephrase our validation theorems with phases such as ``\emph{if $L_- >0$ is sufficiently large, then ...} ''. However as necessitated in constructing a computer assisted proof, we will fix all of our computational parameters, and then after checking explicitly verifiable conditions, determine whether the hypothesis of each validation theorem is satisfied. If any one of the validation criteria prove false, then the entire computer assisted proof is unsuccessful. In this case, we may try to prove the theorem again with a different selection of computational parameters. 

By performing this calculation using validated  numerics, we are able to obtain a mathematical proof of the existence of $\varphi$ and calculate its spectral stability. This methodology is based on interval arithmetic, a fundamental  tool for producing computer assisted proofs in analysis. To briefly overview (see also \cite{rokne2001interval,gomez2019computer}), let us define the set of real intervals $ \mathbb{IR}$ as 
\[
\mathbb{IR} := \{ [a,b] \subseteq \R : -\infty < a \leq b < \infty \}.
\] 
In this manner, a rigorous enclosure of $\sqrt{2}$, for example, can be given in terms of an upper and lower bound. One may define arithmetic operation $ * \in \{ +,-, \cdot , / \}$ on  intervals $ A , B \in \mathbb{IR}$ by $ A * B = \{ \alpha * \beta : \alpha \in A, \beta \in B \}$. These operations are explicitly computable with the exception, of course, of division when $ 0 \in B$. More generally, if $ f : \R^n \to \R^m$ is a continuous function, then we say that $ F: \mathbb{IR}^n \to \mathbb{IR}^m$ is an \emph{interval extension} of $f$ if $f(x) \in F(X)$ for all $ X \in \mathbb{IR}^{n}$ and all $ x \in X$. 

While interval arithmetic provides a computational methodology for performing set arithmetic, some amount of precision is lost. It should be noted that subtraction (division) is not the inverse operation of addition (multiplication);  $[a,b] - [a,b]= [a-b,b-a] \neq [0,0]$  unless $ a=b$. Furthermore the interval extension of a function need not be unique nor sharp. For example the function defined as $ A \mapsto [-1,1]$ for $A \in \mathbb{IR}$ is an interval extension of the sine function. Nevertheless, the forward propagation of error behaves predictably and the final result is guaranteed to contain the correct answer. More precisely, if $A,\tilde{A},B, \tilde{B}  \in \mathbb{IR}$ with $ A \subseteq \tilde{A}$ and $ B \subseteq \tilde{B}$, then $A*B \subseteq \tilde{A}* \tilde{B}$  for $* \in \{ +,-, \cdot , / \}$. This property is sometimes called the inclusion isotony principle.  

On a computer using floating point numbers, one is only able to precisely represent the rational numbers with a finite binary expansion. 
Furthermore, when the sum of two floating point numbers is computed, the sum is necessarily  rounded to an adjacent floating point number. To account for this, when interval arithmetic is implemented on a computer, outward rounding is used in order to ensure that whatever interval arithmetic operation performed will contain the true solution. By controlling \emph{a posteriori} error estimates, interval enclosures may be computed for elementary functions, matrix operations, and even implicitly defined functions, such as the solution to differential equations. 
 
%%%%%%%%%
 
\subsection{Computing a standing wave} \label{sec:ComputeStandingWave}

To calculate a standing wave to \eqref{eq:GradPDE}, we solve the equivalent problem of calculating a heteroclinic orbit in the spatial dynamical system, a Hamiltonian ODE. We define the generalized momentum variable $ w = D v_x \in \R^n$ and consider the following Hamiltonian system: 
\begin{align}
H(v,w) &= \tfrac{1}{2} \|D^{-1} w\|^2 +  G(v).  \\
(v,w)' &= - J \nabla H(v,w) 
\label{eq:HamiltonianODE}
\end{align}
where $J$ denotes the symplectic matrix defined in \eqref{E:eval-sys}. Hence, a standing wave $ \varphi$ to \eqref{eq:GradPDE} corresponds to a heteroclinic orbit $\phi = (\varphi,D \varphi_x)$ to \eqref{eq:HamiltonianODE}.  

Assume that  $p_0 , p_1 \in \R^{2n}$ are equilibria to \eqref{eq:HamiltonianODE}  with equal energy $H(p_0) = H(p_1)$ and without loss of generality, assume that $H(p_0) = 0$. From our assumption (H2) we may assume that the linearizations about $p_0$ and $p_1$ are hyperbolic, both having $n$-stable and $n$-unstable eigenvalues, all distinct. A heteroclinic orbit between these points is an intersection between $W^u(p_0)$, the $n$-dimensional unstable manifold of $p_0$, and $W^s(p_1)$, the $n$-dimensional stable manifold of $p_1$. 
Generically, two $n$ dimensional manifolds embedded in $ \R^{2n}$ will have a trivial intersection. However the assumption $H(p_0) = H(p_1)$ forces both manifolds to live within the same codimension-1 energy surface, whereby they will generically intersect along a 1-dimensional submanifold. 

We compute the heteroclinic orbit by solving a boundary value problem between  $W^u(p_0)$ and $W^s(p_1)$, using validated numerics to compute the (un)stable manifolds and the solution map of the flow. This subject has been well studied in the validated numerics literature, see for example the works  \cite{oishi1998numerical,wilczak2003heteroclinic,van2015stationary,ArioliKoch15,van2018continuation,van2020validated} and the the references cited therein.  We develop bounds on the  (un)stable manifold of a hyperbolic fixed point  using the approach taken in \cite{capinski2017beyond,CapinskiWasieczko-Zajpolhk-ac17,capinski2018computer}. 
As the estimates for the (un)stable manifolds are used in \eqref{eq:Eta_M_C_M} to compute the constants $\eta_M$ and $C_M$ from (H3), we  present the unstable manifold theorem from \cite{capinski2018computer} in full detail. Since this ODE is time reversible, the same estimates follow for bounding the stable manifold of $ p_1$.

To begin our analysis, we perform a linear change of variables to align our coordinate axes with the (un)stable eigenvectors of the linearized flow about $p_0$. In the context of computer assisted proofs, where floating point arithmetic is not perfect, we do not require our change of variables to be perfect. To that end, fix an invertible matrix $ A_0 \in \R^{2n\times 2n}$ whose columns are  (approximate) eigenvectors for the Jacobian  of $- \, J \nabla H(p_0)$. We choose the first $n$ columns of $A_0$ to correspond to the unstable eigenvectors, and the second $n$ columns of $A_0$ to correspond to the stable eigenvectors. We define a change of coordinates 
$ \psi(z) := p_0 + A_0 z$, where $\psi(z) = (v, w)^T$, and correspondingly define  a conjugate dynamical system by: 
\begin{align}
	f (z) &= - A_0^{-1} J \nabla H ( \psi(z)) \\
	z' &= f(z) .\label{eq:ConjugateODE}
\end{align}
If the columns of $A_0$ are chosen to be the exact eigenvectors, then the unstable and stable  eigenvectors of  the linearization of \eqref{eq:ConjugateODE} at $0$ are given by basis vectors $ \{ e_i\}_{i=1}^n$ and $ \{ e_i\}_{i=n+1}^{2n}$ respectively. 

Locally, the (un)stable manifolds of the origin can be well approximated as a graph over its (un)stable eigenspaces. Here we follow the approach taken in \cite{capinski2017beyond} to approximate these invariant manifolds, identify a neighborhood of validity, and provide rigorous error estimates using  validated numerics. As precisely detailed in Theorem \ref{prop:UnstableManifold_Summary}, if certain \emph{a posteriori} conditions are satisfied (depending on fixed constants $ r_u,r_s >0$) then the unstable manifold to \eqref{eq:ConjugateODE} is contained within the set $ S = B_u  \times B_s $,   where 
\begin{align}\label{eq:BallDef}
B_u  &= \left\{ (x,0) \in \R^{2n} : |x| \leq r_u \right\} ,
&
B_s  &= \left\{ (0,y) \in \R^{2n} : |y| \leq r_s \right\} .
\end{align} 
These \emph{a posteriori} conditions are satisfied if $(i)$  the flow exits $S$ through  the set $ \partial B_u \times B_s$; $(ii)$  the flow enters $S$ through  the set $ B_u \times \partial B_s$; and $(iii)$  the flow stays inside a cone of angle $\LL = r_s / r_u$. See Figure \ref{fig:UnstableManifold} for a pictorial representation of these conditions.  

\begin{figure}
	\centering
	\includegraphics[width=0.9\linewidth]{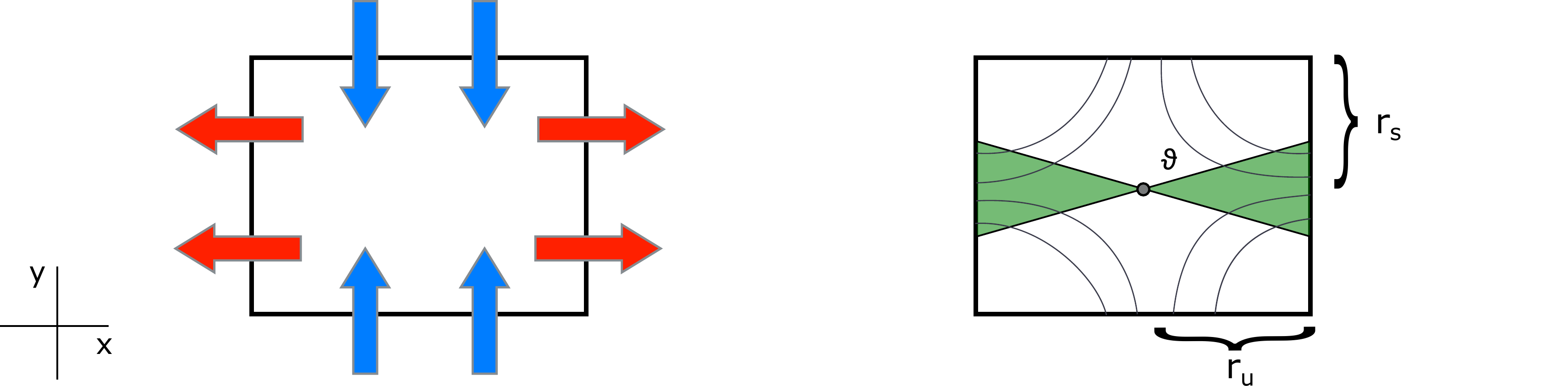}
	\caption{Conditions to check when using a computer to prove an unstable manifold theorem. The top and bottom sides of the box correspond to $B_u \times \partial B_s$ and the left and right sides correspond to $ \partial B_u \times B_s$.}
	\label{fig:UnstableManifold}
\end{figure}

This last condition is checked using logarithmic norms \eqref{eq:DefLogNorm} and logarithmic minimums \eqref{eq:DefLogMin} which, for the Euclidean  norm on $\R^N$, is defined for square matrices $ A\in \mathbb{R}^{N\times N}$ as follows:
\begin{align} \label{eq:DefLogNorm}
l(A) &:= \max \left\{ 
\lambda \in \mbox{ spectrum of } (A+A^T)/2  
\right\}, \\
%%%%%%%%%%%%
\label{eq:DefLogMin}
m_l(A) &:= \min \left\{ 
\lambda \in \mbox{ spectrum of } (A+A^T)/2  
\right\} .
\end{align}
The logarithmic norm $ l(A)$ and logarithmic minimum $ m_l(A)$ may be positive or negative, and provide upper and lower bounds on the linear flow defined by a matrix $A$, that is  $e^{m_l(A)t} \leq \| e^{At} \| \leq e^{l(A) t } $. For further details, we refer the reader to \cite{hairer1987solving,capinski2017beyond} and the references therein. 

For a set $ S \subseteq \R^{2n}$  we define interval matrices 
 $	  	 \left[ \pi_x \frac{\partial f}{\partial x} (S) \right], \left[  \pi_x \frac{\partial f}{\partial y} (S) \right] \in \mathbb{IR}^{n\times n} $ by 
\begin{align*}
	\left[  \pi_x \tfrac{\partial f}{\partial x} (S) \right] &:= 
	\left\{
	A = (a_{i,j}) \in \mathbb{IR}^{n\times n}
	:
	a_{i,j}  = 
	\left[
	\inf_{z \in S} 
	\pi_{x_i} \tfrac{\partial f}{\partial x_j} (z) ,   	 \sup_{z \in S} 
	\pi_{x_i} \tfrac{\partial f}{\partial x_j} (z) 
	\right]
	\right\} \\
	%%%%%%%%%%%%%%%%%%%
	\left[  \pi_x \tfrac{\partial f}{\partial y} (S) \right] &:= 
	\left\{
	A = (a_{i,j}) \in \mathbb{IR}^{n\times n} 
	:
	a_{i,j}  = 
	\left[
	\inf_{z \in S}
	\pi_{x_i}  \tfrac{\partial f}{\partial y_j} (z) ,   	
	\sup_{z \in S} 
	\pi_{x_i} \tfrac{\partial f}{\partial y_j} (z) 
	\right]
	\right\} .
\end{align*}
Here, for a point $z = (x,y) \in \R^{2n}$ the maps $z \mapsto 	\pi_{x_i}(z)$ and $z \mapsto 	\pi_{y_i} (z)$ denote the projections onto the $i$\textsuperscript{th} components of $x$ and $y$ respectively.  
These sets $	\left[  \pi_x \tfrac{\partial f}{\partial x} (S) \right]$ and $ 	\left[  \pi_x \tfrac{\partial f}{\partial y} (S) \right] $are essentially a rectangular, convex closure of the $ \pi_x \frac{\partial f}{\partial x}$ and $\pi_x \tfrac{\partial f}{\partial y} $~images of a set $S$. 
We now present the unstable manifold theorem from \cite{capinski2018computer}. 
\begin{Theorem}[{\cite[Theorem 12]{capinski2018computer}}] \label{prop:UnstableManifold_Summary}
 Suppose $p_0 \in \R^{2n}$ is an equilibrium of \eqref{eq:HamiltonianODE}. 
 Fix $r_u , \LL>0$ and $ r_s = \LL r_u$. 
 Define the sets $B_u $ and $B_s$ as in \eqref{eq:BallDef} 
and  $ S  = B_u \times B_s \subseteq \R^{2n} $. 
Writing     $(x,y) \in B_u \times B_s$, we define rate constants $ \mu , \xi \in \R$  as follows 
 \begin{align}
\mu &= \sup_{z \in S} \left\{
l\left(  
\pi_y \tfrac{\partial f}{\partial y}  (z) 
\right)
+ \frac{1}{\LL} 
\left\| 
\pi_y
\tfrac{\partial f}{\partial x}  (z) 
\right\|
\right\},\\
%%%%%%%%%%%%%%%%%%%% 
%%%%%%%%%%%%%%%%%%%%  
\label{eq:Def_xi}
\xi &=   
\inf_{ A \in  \left[ \pi_x \frac{\partial f}{\partial x} (S) \right] } m_l\left(A
\right)
- 
\sup_{ A' \in  \left[\pi_x \frac{\partial f}{\partial y} (S) \right] } 
\LL  
\left\| A' 
\right\| ,
\end{align}
% \begin{align}
% \mu &= \sup_{z \in D} \left\{
% l\left(
% \frac{\partial}{\partial y} f_y(z) 
% \right)
% + \frac{1}{L} 
% \left\| 
% \frac{\partial}{\partial x} f_y(z) 
% \right\|
% \right\},\\
% %%%%%%%%%%%%%%%%%%%% 
% %%%%%%%%%%%%%%%%%%%%  
% \label{eq:Def_xi}
% \xi &=   
%\inf_{ A \in  \left[ \frac{\partial f_x}{\partial x} (S) \right] } m_l\left(A
% \right)
% - 
% \sup_{ A' \in  \left[ \frac{\partial f_x}{\partial y} (S) \right] } 
% L  
% \left\| A' 
% \right\| ,
% \end{align}
%\MBnew{Also, it seems the notation is not consistent in the above two equations. In one there is the quantity $\pi_y \partial f/\partial x(z)$, with $z \in D$ (or really $z  \in S$), and in the other it is written just as $A$, with $A \in [\pi_y \partial f/\partial x(S)]$} \JJ{I copied the notation from \cite{capinski2018computer}. From the point of interval arithmetic, there is a subtle difference between these two quantities. }
% \JJ{Explain how to interpret $f_x(D)$}
 
Suppose the following conditions are satisfied:
 	\begin{enumerate}
 	\item For any $q \in \partial B_u \times B_s$ we have $\left< \pi_x f(q) , \pi_x q \right> >0$; 
 	\item For any $q \in  B_u \times \partial B_s$ we have $\left< \pi_y f(q) , \pi_y q \right> <0$;  
 	\item $\mu < 0 < \xi $. 
 \end{enumerate}
Then there exists a unique $C^1$-function $ \omega^u : B_u \to B_s$ such that $\| \partial_x \, \omega^u \| \leq L$, and 
 \[
 W^u_{loc}(0) = \left\{
 (x,\omega^u(x))  : x \in B_u
 \right\}
 \]
 is the local unstable manifold of $0$ for the differential equation \eqref{eq:ConjugateODE}. 
 Moreover the map $ P:B_u \to \R^{2n}$ defined by $ P(x) := \psi(x,\omega^u(x))$  
 is a chart for  $W^u_{loc}(p_0)$, the local unstable manifold of $p_0$ for the differential equation \eqref{eq:HamiltonianODE}.

\end{Theorem} 
The intuition for the three conditions above, is that we require: $(i)$  the flow exits through the ``unstable faces''  of the product $B_u \times B_s$;  $(ii)$ the flow enters $S$ through  the ``stable faces'' of the product $B_u \times B_s$; and $(iii)$  the flow stays inside a cone of angle $\LL = r_s / r_u$, analogous to the rate constants of Fenichel theory.   

The hypotheses of Theorem \ref{prop:UnstableManifold_Summary} can be explicitly verified using interval arithmetic. Of particular interest is the constant $\xi$ which we use in \eqref{eq:DecayAtInfinity} to define the constant $\eta_M$ from (H3). Note by reversing time, Theorem \ref{prop:UnstableManifold_Summary} can be used to provide explicit bounds on the stable manifold $ W^s_{loc}(p_1)$. 

To compute a heteroclinic orbit to \eqref{eq:HamiltonianODE}, we solve a boundary value problem between $W^u(p_0)$ and $ W^s(p_1)$. That is, we seek to find points $ q_0 \in W^u(p_0)$ and $ q_1 \in W^s(p_1)$ such that if we integrate $ q_0$ forward along the flow and we integrate $ q_1$ backwards along the flow, they will meet in the middle. 
That is $\Phi_T(q_0) = \Phi_{-T}(q_1)$ for some $ T>0$, where $ \Phi_T$ denotes the time $T-$solution map of \eqref{eq:HamiltonianODE}. To that end, we define in \eqref{eq:HeteroclinicBVP}  below a function $ h : \R^{2n} \to \R^{2n}$ so that its zeros are (generically) isolated, and correspond to a heteroclinic orbit $\phi: \R\to \R^{2n}$ connecting $p_0$ and $p_1$. 

Recall that $ W^u(p_0)$ and $W^s(p_1)$ will generically have  a transverse intersection only  within the $0-$energy surface $ \cH^0 = \{ (v,w) \in \R^{2n} : H(v,w) =0\}$. Away from singular points, $ \cH^0$ can be locally written as a graph over its first $ 2n-1$ coordinates; 
if $(v,w) \in \cH^0$ and $ w_n > 0$, then $w_n =   d_n \sqrt{- \sum_{k=1}^{n-1} | d_k^{-1} w_{k} |^2 - 2 G(v) }$. 
Hence, in order to check that $\Phi_T(q_0) = \Phi_{-T}(q_1)$, it suffices to check that the two points agree on their first $ 2n-1$ coordinates, and their $ w_n$ coordinates are of the same sign.  

To explictly define our boundary value problem operator $h$, 
suppose that the hypotheses of Theorem \ref{prop:UnstableManifold_Summary} are satisfied for some $r_u>0$ and $ B_u^0 \subseteq \R^n$, whereby  $P: B_u^0 \to \R^{2n}$  is  a chart for the (local)  unstable manifold of $p_0$.  
Similarly, for $r_s>0$ and $ B_s^1  \subseteq  \R^n$, suppose  $Q: B_s^1 \to \R^{2n}$ is a chart for the (local)  stable manifold of $p_1$.   
A heteroclinic orbit between $ p_0$ and $ p_1$  can  be identified by finding some $ q_0 \in B_u^0$, $ q_1 \in B_s^1$ and $ T>0$ such that $ \Phi_T(P(q_0)) = \Phi_{-T}(Q(q_1))$. 
While the heteroclinic orbit may be unique, the same may not be said of $q_0$ and $q_1$.  
To impose a phase condition, we fix $T>0$ and $r_u >0$, and  require that $ \|q_0\| = r_u$. 
Thus, we define our boundary value operator 
$ h:B_u^0 \times B_s^1  \to \R^{2n} $
 as follows: 
\begin{align} \label{eq:HeteroclinicBVP}
\left[
h(q_0,q_1)
\right]_i &:= 
\begin{cases}
\left[
\Phi_{T}(P(q_0)) - 
\Phi_{-T} (Q(q_1))
\right]_i
 %%%%%%%%%%%%%%% 
& \mbox{ if } 1\leq i \leq 2n-1,
   \\
 \|q_0\|^2 - r_u^2 &\mbox{ if } i =2n 
\end{cases} 
\end{align} 
Since $p_0$ and $p_1$ have equal energy, if $h(q_0,q_1) =0$ and the $2n$\textsuperscript{th} component of both  $\Phi_{T}(P(q_0)) $ and $\Phi_{-T} (Q(q_1))$ are of the same sign, 
then these component values are equal and moreover $\Phi_{T}(P(q_0)) =  \Phi_{-T} (Q(q_1)) =: \phi(0)$. 
It then  follows that $ \phi(x):= \Phi_{x}(\phi(0))$ is a heteroclinic orbit with limits: $  \lim_{x \to - \infty} \phi(x) = p_0$ and $\lim_{x \to +\infty} \phi(x) = p_1$.  

Thus, proving the existence of a locally unique heteroclinic orbit between equilibria $p_0$ and $p_1$ is essentially reduced to proving the existence of locally unique solutions to $0 = h(q_0,q_1)$ as defined in  \eqref{eq:HeteroclinicBVP}. 
For a given domain $U\subseteq \mathbb{IR}^{2n}$, we use the CAPD library \cite{kapela2020capd} to calculate rigorous enclosures of $ \Phi_T  (U)$ and $D \Phi_T  (U)$.
Lastly, we are able to prove that an explicit neighborhood $U \subseteq \mathbb{IR}^{2n}$  contains a unique zero of $h$ by using the interval-Krawczyk method  \cite{hansen2003global}.

\subsection{Determining $L_-$}
\label{sec:L-Calculation}

As described in Section 
\ref{S:L_-}, 
we wish to find $L_- >0$ such that 
$\mathbb{E}_-^u(x; 0) \cap \mathcal{D} = \{0\}$  for all $ x \leq - L_-$. 
Conditions for this are established in Proposition \ref{prop:L-}, namely if \eqref{E:L-condition} is satisfied. 
This inequality depends on a nontrivial constant $ \lambda_{\cT_-}$.

In Proposition \ref{prop:L-} we define the $\lambda_{\mathcal{T}_-} :=  \frac{K_-C_M}{\eta_M}e^{-\eta_M L_-}$  to bound the error when computing an eigenfunction, cf  \eqref{E:minus-infty-asymptotics}. 
The constants $C_M$ and $\eta_M$, as we recall from Hypothesis (H3), are chosen such that $|\mathcal{M}(x) - \mathcal{M}_\pm| \leq C_M e^{-\eta_M|x|}$. Thus by choosing $ L_-$ large we may ensure $|\mathcal{M}(x) - \mathcal{M}_\pm| \ll 1$. 
	
In the present application, we define  $ \cM(x) = \nabla^2G( \varphi(x)) $	relative to a heteroclinic orbit   $\phi = (\varphi,D \varphi_x)$ to \eqref{eq:HamiltonianODE}. 
We note, however, that  heteroclinic orbits are only unique up to translation; if $\phi:\R \to \R^{2n}$ is a heteroclinic orbit, then so too is $ \phi_\sigma(x) := \phi(x+\sigma)$ for any  $ \sigma \in \R$. 
In this sense the choice of $L_-$ is somewhat arbitrary, or said differently, highly dependent on the choice of translation $\sigma$. 

As a result, when bounding  $\lambda_{\cT_-}$ in Theorem \ref{prop:Bounding_L-} below, it is not the particular value of $ L_-$ but instead the properties of $ \phi(-L_-)$ which have a direct effect on our estimates. 
%Correspondingly,  in Theorem \ref{prop:Bounding_L-} below, it is not the particular value of $L_-$, but instead the properties of $ \phi(L_-)$ which has a direct effect on the bound for    $\lambda_{\cT_-}$. 
In our hypothesis we fix $ \phi(-L_-) \in W^u_{loc}(p_0) \subseteq B_u \times B_s$, where $B_u$ has a radius of $r_u$ and $B_s$ has a radius of $r_s = \vartheta r_u$.  
By taking $r_u $ small, we are able to control the size of  $\lambda_{\cT_-}$. 

\begin{Theorem} \label{prop:Bounding_L-}
Assume the hypothesis of Theorem \ref{prop:UnstableManifold_Summary}, having fixed appropriate constants $ r_u , \vartheta, \xi >0$, and define $K_-$ as in \eqref{eq:K_-def}. 
Suppose that $ \phi$ is a heteroclinic orbit to \eqref{eq:HamiltonianODE} such that   $ \phi(-L_-) \in W^u_{loc}(p_0)$.  
Then 
\[
\lambda_{\cT_-}  
\leq 
\xi^{-1} 
K_- C_G 
r_u \big\|   A_0 
\big\|  
 \sqrt{1+\LL^2}  
,
\]
where\footnote{Note that $D^3G(\pi_1 \circ \psi(w))$ is a rank 3 tensor. We may estimate the spectral norm of a rank-3 tensor $A = (a_{ijk})$ by defining matrices $A_k$ as  $(A_k)_{ij} = a_{ijk}$ and estimating  $\| A \| \leq \left( \sum_{k} \| A_{k}  \|^2 \right)^{1/2}$.}
\[
C_G \leq  \sup_{z \in S} 
\|    D^3 G( \pi_1 \circ \psi(z))  \| .
\]  
\end{Theorem}
\begin{proof}
	We first calculate constants $ C_M,\eta_M$ defined in \eqref{E:A-bound} for which 
	\[
	\|\mathcal{A}_-(x)\| \leq C_M e^{-\eta_M |x|}, \qquad x \leq 0.
	\]
	Note that this bound is only needed for $ x \leq - L_-$. 
	We compute a bound on $\cA_-(x)$ below:
	\begin{align*}
\left\|	\cA_-(x) \right\|&= 
\left\|	\cM_- - \cM(x) \right\|\\
	%%%%%%%%%%%%%%
	&=\left\| D^2G( \varphi_-) -  D^2G( \varphi(x)) \right\| ,
	\end{align*} 
	where $ \varphi_- = \pi_1 p_0$. 
		Recall $ \phi = ( \varphi, D\varphi_x)$, where $D $ is the diffusion matrix.
	By Theorem \ref{prop:UnstableManifold_Summary} we have $ \psi^{-1} \circ \phi(x) \in S = B_u \times B_s$ for all $ x \leq -L_-$.  From the definition of $ C_G$ and the mean value theorem  it follows that 
	\[
	\|\cA_-(x)\| \leq C_G \| \varphi(x) - \varphi_-\| . 
	\]
	To bound $\| \varphi(x) - \varphi_-\|$, we note 
from   \cite[Theorem  14]{capinski2018computer}   that 
	\begin{align} \label{eq:DecayAtInfinity}
	\| \psi^{-1} \circ \phi(-(x+L_-)) \| 
	&\leq 
     r_u
     	\sqrt{1+\LL^2} 
	e^{-\xi |x|}
	, \quad\quad x\geq 0.
	\end{align}	
	Hence  for $ x \leq - L_-$ it follows that 
	\[
	\| \phi(x) - p_0 \| 
	\leq r_u \|A_0\| \sqrt{1+\LL^2} 
	e^{-\xi | x+L_-|} \quad \quad x \leq -L_-.
	\]
	Thereby, if we define:
	\begin{align}
		\eta_M &= \xi, 
		&
		C_M  &= C_G r_u \|A_0\|  \sqrt{1+\LL^2} 
		e^{\xi L_-}, \label{eq:Eta_M_C_M}
	\end{align}
	then \eqref{E:A-bound}  is satisfied for $ x \leq -L_-$. 
		From Proposition \ref{prop:L-contraction} we have 
	\[
	\lambda_{\cT_-}  
	\leq 
	\frac{K_- C_M }{\eta_M} e^{-\eta_ML_-} 
	=
	\xi^{-1} 
	K_- C_G 
	r_u \big\|   A_0 
	\big\|  
	\sqrt{1+\LL^2}  .
	\]
%	Our estimate on $	\lambda_{\cT_-}$ follows. 
\end{proof}

\subsection{Calculating $ \bE_-^u(x)$}
\label{sec:CalculatingE-u}

 Recall our definition  of $\mathbb{E}^u_-(x)$ as being the subspaces of solutions to \eqref{E:main-sys} which are asymptotic to $\mathbb{E}^u_{-\infty}$. 
% In Section \ref{S:L_-} we developed bounds on functions $ W_j^{-,u}:  (-\infty,-L_-] \to \R^{2n}$ for which 
% \[
% 	\mathbb{E}^u_-(x) = \mbox{span}\{ W_1^{-,u}(x) , \dots , W_n^{-,u}(x) \}, \qquad\qquad x \leq -L_-.
% \] 
 That is, if $ U_i(x): \R \to \R^{2n}$ are linearly independent solutions to    \eqref{E:main-sys} for $ 1 \leq i \leq n$ such that $\lim_{x \to - \infty} U_i(x) = 0$, 
 then
 \[
 \mathbb{E}^u_-(x) = \mbox{span}\{ U_1(x) , \dots , U_n(x) \}, \qquad\qquad  \forall x \in \R.
 \] 
Recall also our results from Proposition \ref{prop:L-contraction} and  Proposition \ref{prop:L-},  that % then  
$\mathbb{E}^u_-(x)$ may be represented by a frame matrix $ W^{-,u}(x) $ for $ x \leq -L_-$  with columns satisfying 
\[
 \|W_j^{-,u}(\cdot) - V^{-,u}_j \|_X \leq \frac{\lambda_{\mathcal{T}_-}}{1-\lambda_{\mathcal{T}_-}}|V_j^{-,u}| =: \epsilon_j,
 \qquad \qquad x \leq -L_-.
\]
where $X = L^\infty((-\infty, -L_-], \R^{2n} )$. 
 To compute a frame matrix of $\mathbb{E}^u_-(x)$ for $x \geq -L_-$ one   has to solve \eqref{E:main-sys} taking as initial values the columns   $ W_j^{-,u}(-L_-)$. 
  A difficulty here is that we only have a  basis of $\mathbb{E}^u_-(-L_-)$ accurate up to a certain error, and this error needs to be propagated forward when solving the IVP \eqref{E:main-sys}.

  To keep track of error, for $ \cV^{-,u} = \left[ V_1^{-,u} | \dots | V_n^{-,u}  \right]$ fix a numerical approximation of the eigenvectors 
  $ \bar{\cV}^{-,u} = \left[ \bar{V}_1^{-,u} | \dots | \bar{V}_n^{-,u}  \right]$ and fix $ \bar{\epsilon}_i$ such that $ | V_i^{-,u} - \bar{V}_i^{-,u}  | \leq \bar{\epsilon}_i$.    
  Define  $E(x) : ( - \infty ,-L_-] \to   \mathbb{R}^{2n\times  n}$    such that $ W^{-,u}(x)  = \bar{\cV}^{-,u} + E(x)$.
It follows that the columns satisfy  
\begin{align*}
| E_i(x) | &\leq |W_i^{-,u}(x) - V^{-,u}_i | + | V_i^{-,u} - \bar{V}_i^{-,u}  | \\
&\leq \epsilon_i + \bar{\epsilon}_i
\end{align*}
for $1 \leq i \leq n$ and $ x \leq -L_-$. 
To bound $E(x)$ computationally we   construct an interval matrix $ \cE^- \in \mathbb{IR}^{2n\times  n }$  below: 
\begin{align} \label{eq:L_-Error}
\cE^- &:= 
\left[
\begin{array}{c|c|c}
\null \left[-\epsilon_1-\bar{\epsilon}_1,\epsilon_1+\bar{\epsilon}_1\right] &   & \left[-\epsilon_n-\bar{\epsilon}_n,\epsilon_n+\bar{\epsilon}_n\right] \\
\vdots  & \hspace{2mm}\cdots\hspace{2mm} & \vdots \\
\null \left[-\epsilon_1-\bar{\epsilon}_1,\epsilon_1+\bar{\epsilon}_1\right] &   & \left[-\epsilon_n-\bar{\epsilon}_n,\epsilon_n+\bar{\epsilon}_n\right] 
\end{array}
\right] 
.
\end{align} 
Note then that $ E(x) \in \cE^-$ and therefore    $ W^{-,u}(x) \in \bar{\cV}^{-,u} +  \cE^-$ for $ x \leq -L_-$.   
To bound a frame matrix of  $\mathbb{E}^u_-(x)$ for $ x \geq -L_-$, 
it suffices to bounds solutions of \eqref{E:main-sys} with the initial values $\tilde{W}_i^{-,u}(-L_-) = \bar{V}_i^{-,u} + \tilde{E}_i$ for each $ 1 \leq i \leq n$ and  all possible initial data $ \tilde{E}_i  \in \cE_i^-$. 
This may be accomplished with standard techniques in validated numerics, eg \cite{kapela2020capd}. 

Note first that any error terms of the initial condition in the unstable eigendirections will grow exponentially.  
To mollify this problem  we perform a  change of variables; 
for $ x = -L_-$ 
there exist matrices $E_u(-L_-),E_s(-L_-) \in \R^{n\times n}$ such that 
\[
E(-L_-)
=  \bar{\cV}^{-,u}  \cdot E_u(-L_-)  + \bar{\cV}^{-,s}  \cdot E_s(-L_-) 
 .   
\] 
From its definition in \eqref{eq:L_-Error} it follows that 
$
\cE^- \ni 
 \bar{\cV}^{-,u}  \cdot E_u(-L_-)  + \bar{\cV}^{-,s}  \cdot E_s(-L_-) 
$.
Hence  interval enclosures $ \cE_u , \cE_s \in \mathbb{IR}^{n\times n }$  containing
$E_u(-L_-)$ and $ E_s(-L_-)$ may be computed 
as follows
\[
\left[
\begin{array}{cc}
\cE_u \\
\hline
\cE_s
\end{array}
\right]
:=
\left[
\begin{array}{c|c}
\bar{\cV}^{-,u} & \bar{\cV}^{-,s}
\end{array}
\right]^{-1}
\cE^{-} .
\] 
If $I + E_u(-L_-)$ is invertible for the $n\times n$ identity matrix $I$, then a new frame matrix $U(-L_-) \in \mathbb{R}^{2n \times n}$ for $\bE_-^u(-L_-)$  may be defined as follows 
\begin{align} \label{eq:NumericalFrame_Center}
U(-L_-) &:= W^{-,u}(-L_-) \cdot \left(I + E_u(-L_-)  \right)^{-1} \\
&= \left(\bar{\cV}^{-,u} + E(-L_-) \right) \cdot \left(I + E_u(-L_-)\right)^{-1} 
\nonumber 
\\
&= \bar{\cV}^{-,u}  + \bar{\cV}^{-,s}  \cdot E_s(-L_-) \left(I + E_u(-L_-) \right)^{-1} 
 \nonumber 
.
\end{align}
Hence if $I + \tilde{E}_u$ is invertible for all $ \tilde{E}_u \in \cE_u$, 
then  we may define an interval enclosure  $\cU(-L_-) \in \mathbb{IR}^{2n \times n}$  
of the frame matrix $U(-L_-)$ as  follows 
\[
\cU(-L_-) :=  \bar{\cV}^{-,u}  + \bar{\cV}^{-,s}  \cdot \cE_s \left(I + \cE_u \right)^{-1} .
\]
Note that all of the error terms in $ \cU(-L_-)$ are essentially  in the stable eigendirections.

Thus for any $L_+ \geq 0$ we may  define a frame matrix $U:[-L_-,L_+] \to \R^{2n \times n}$ for $\bE_-^u(x)$ by defining the $i^{th}$ column of  $U(x)$
as the solution to the initial value problem \eqref{E:main-sys} taking as an initial condition at $x=-L_-$ the $i^{th}$ column of $U(-L_-)$.

Using rigorous numerics, we are able to similarly extend $\cU(-L_-)$ to a function $ \cU: [-L_- ,L_+] \to \mathbb{IR}^{2n \times n}$ and obtain an enclosure of the frame matrix $U(x)$ for $ x \in [-L_- ,L_+]$. 
In particular,  we define the $i^{th}$ column of  $\cU(x)$  as the solution to the initial value problem \eqref{E:main-sys} taking as an initial condition at $x=-L_-$ the $i^{th}$ column of $\cU(-L_-)$.

\subsection{Determining $L_+$}
 \label{sec:CalculationForL+}
 
To prove that 
$\mathbb{E}_-^u(x; 0) \cap \mathcal{D} = \{0\}$  for all $ x \geq  L_+$, we aim to show that $ \| \pi_{\Gamma_1} - \pi_{\Gamma_2}\| < 1 $ using the estimate in \eqref{eq:ProjectionMatrixBound}. 
Analogous to Theorem \ref{prop:Bounding_L-}, we may explicitly obtain bounds on $\lambda_{\cT_+}$ and define $\epsilon_0$ as in \eqref{E:def-epsilon0}.
% so that $ | \tilde{W}_j^{+,u/s} (x) | \leq \epsilon_0$ for all $j,u,s$ and $ x \geq L_+$. \JJ{Give more intuition for what is being computed.} 
To obtain all of the necessary constants for Proposition \ref{prop:main2}, we need to compute bounds on 
\begin{align} \label{eq:epsB_muStar}
	\epsilon_b &:=\frac{1}{\sqrt{\mu^*}} e^{-2 \nu_n^+ L_+ } \| \tilde{B} \|, 
&
\mu^* &:= \min\{ \mu: \mu \in \sigma(\tilde \Gamma^T \tilde \Gamma)\},
\end{align}
The matrices $\tilde{B}, \tilde{\Gamma}  $ are defined in terms of how a frame matrix $U(0)$ of $ \mathbb{E}_-^u(0)$   may be written in terms of basis vectors spanning $\mathbb{E}^s_+(0)$ and $\mathbb{E}^u_+(0)$, 
 cf  \eqref{E:soln-exp-2} and \eqref{E:Gamma2-frame}. 
   By increasing $L_+$ one is able to control the size of $ \epsilon_b$, however   the precise value of $L_+$ should be taken with a grain of salt as heteroclinic orbits are translation invariant.     
   As Theorem \ref{prop:Bounding_L-} was stated without explicit reference to the value of $L_-$,  we will without loss of generality choose $ L_+ =0$. 

Before further discussing how to compute the matrices $\tilde{B}, \tilde{\Gamma} $,  recall that throughout Section \ref{S:L_+} we assume the particular form of the frame matrix $U(x)$ of $ \mathbb{E}_-^u(x)$ given  in \eqref{E:subspace-rep}. 
Namely, we assume that $U(x) = [ U_\varphi |  U_1 | \dots |  U_{n-1}]$ where $ U_{\varphi} = ( \varphi',D\varphi'')$, and $ U_1 , \dots ,  U_{n-1}$ are some other vectors for which  the columns of $U(x)$ span $ \mathbb{E}_-^u(x)$.  
%This choice for the column vectors $ U_1, \dots , U_{n-1}$  is unique up to a change of basis; eg. $   \{U_{k}\}  \mapsto  \{U_k A\} $ for $ A \in GL_{n-1}(\R)$. \JJ{Make this correct.} 
However, this choice of frame matrix is generically at odds with how we defined the frame matrix $U(x)$ in \eqref{eq:NumericalFrame_Center}, where for $ 1 \leq i \leq n $ we simply  took each column $U_i(-L_-)$ to be an approximate unstable eigenvector of  $\mathcal{A}_{-\infty}$.

To define a frame matrix of the form required in \eqref{E:subspace-rep}, first fix some $1 \leq i \leq n$, and define a matrix $ 	U_{\hat{i}}(x) \in \R^{2n \times (n-1)}$ by 
\begin{align}
	U_{\hat{i}}(x) &:= 
	\left[
	\begin{array}{c|c|c|c|c|c}
		U_1(x) & \dots & U_{i-1}(x) & U_{i+1}(x) &  \dots & U_{n}(x)
	\end{array}
	\right] 
\end{align}
where each $U_k$ is a column of the frame matrix $U$ defined in \eqref{eq:NumericalFrame_Center}. 
So long as $ U_{\varphi}$ is not in the span of the columns of $ 	U_{\hat{i}}(x)$ then $\tilde{U}(x) := [ U_\varphi(x) | U_{\hat{i}}(x)]$ is a frame matrix for $ \bE_-^u(x)$.

For this particular choice of frame matrix, we may now compute the matrices $B, \Gamma \in  \mathbb{R}^{n \times n-1}$. 
Recall from  \eqref{E:soln-exp-2} that  $\beta_j^k = \tilde \beta_j^k e^{-\nu_j^+ x}$,  $\gamma^k_j = \tilde \gamma_j^k e^{\nu_j^+ x}$, and 
\begin{align*}
	U_k(x) &= \sum_{j=1}^n  \tilde{\gamma}_j^k  \tilde{V}_j^{+,u} +  \tilde{\beta}_j^k \tilde{V}_j^{+,s}. \\
	&= \sum_{j=1}^n  {\gamma}_j^k  (V_j^{+,u} + \tilde{W}_j^{+,u}(x)) +  {\beta}_j^k  (V_j^{+,s} + \tilde{W}_j^{+,s}(x)) . 
\end{align*}
Note from \eqref{E:W-error} that $|\tilde{W}_j^{+, s/u}(x)| \leq \epsilon_0$ for all $x \geq L_+$ and all $  j, s, u.$ 
% \JJ{Changed to centering about approximate e-vectors.}
Thus, having assumed $L_+=0$,  fixed $1 \leq i \leq n$, and  writing $\tilde{B} = [\tilde{\beta}^1 | \dots| \tilde{\beta}^{n-1}  ]$ and $\tilde{\Gamma} = [\tilde{\gamma}^1 | \dots| \tilde{\gamma}^{n-1}  ]$ as in \eqref{E:BGtilde}, 
we obtain 
\[
U_{\hat{i}} (L_+) = \left(
\left[
\begin{array}{c|c}
\cV^{+,u} & \cV^{+,s} 
\end{array}
\right] 
+
\left[
\begin{array}{c|c}
	\cW^{+,u} & \cW^{+,s} 
\end{array}
\right] 
\right)
\left[
\begin{array}{c}
\tilde{\Gamma}  \\
	\hline 
	\tilde{B}
\end{array}
\right] .
\]
Let $\bar{\cV}^{+,u} ,  \bar{\cV}^{+,s} $ denote numerical approximations of $\cV^{+,u} ,  \cV^{+,s} $ and define an interval matrix 
$\mathcal{E}_{L_+} \in \mathbb{IR}^{2n \times 2n}$ such that 
\[
\mathcal{E}_{L_+} \ni 
%%%
\left[
\begin{array}{c|c}
	\cW^{+,u} & \cW^{+,s} 
\end{array}
\right] 
+
%\left(
\left[
\begin{array}{c|c}
	\cV^{+,u} &\cV^{+,s} 
\end{array}
\right] 
-
\left[
\begin{array}{c|c}
	\bar{\cV}^{+,u} & \bar{\cV}^{+,s} 
\end{array}
\right],
%\right)
\]
from which it follows that 
\begin{align} \label{eq:GammaBetaCalc}
	\left[
	\begin{array}{c}
		\tilde{\Gamma}  \\
		\hline 
		\tilde{B}
	\end{array}
	\right] 
&	\in \left( I
	+
	\left[
	\begin{array}{c|c}
		\bar{\cV}^{+,u} &\bar{\cV}^{+,s} 
	\end{array}
	\right]^{-1}
	\mathcal{E}_{L_+}
	\right)^{-1}
	\left[
	\begin{array}{c|c}
		\bar{\cV}^{+,u} &\bar{\cV}^{+,s} 
	\end{array}
	\right]^{-1}
	U_{\hat{i}} (L_+).
\end{align}
Using validated numerics we are able to solve \eqref{eq:GammaBetaCalc} to obtain enclosures of $\tilde{\Gamma} \in  \tilde{\Gamma}_I\in \mathbb{IR}^{n\times n-1}$ and   $\tilde{B} \in  \tilde{B}_I\in \mathbb{IR}^{n\times n-1}$.

From here it is straightforward to bound $\| \tilde{B}\|$ in terms of the interval enclosure $ \tilde{B}_I$ needed in the estimate for $ \epsilon_b$. However obtaining a lower bound on $\mu^*$, cf \eqref{eq:epsB_muStar}, requires slightly more finesse.  
First, let us write $ \tilde{\Gamma}_I = \tilde{\Gamma}_c + \tilde{\Gamma}_\Delta$, where $ \tilde{\Gamma}_c \in \mathbb{IR}^{n\times n-1}$ is the midpoint of $ \tilde{\Gamma}_I$ with zero width and $ \tilde{\Gamma}_\Delta = \tilde{\Gamma}_I - \tilde{\Gamma}_c$. 
We obtain a bound on $ \mu^*$ by way of the Rayleigh quotient.
\begin{align*}
\mu^* &= \min \{ \mu : \mu \in  \sigma( \tilde{\Gamma}^T \tilde{\Gamma}) \} \\
%&\geq \inf_{\Gamma \in \tilde{\Gamma}_I} \inf_{x \in \R^{n-1}; |x|=1}
%\frac{x^T \Gamma^T \Gamma x}{x^T x} \\
&\geq \inf_{\tilde{\Gamma}_\delta \in \tilde{\Gamma}_\Delta}
\inf_{x \in \R^{n-1}; |x|=1}
\frac{x^T (\tilde{\Gamma}_c+\tilde{\Gamma}_\delta)^T (\tilde{\Gamma}_c+\tilde{\Gamma}_\delta) x}{x^T x}\\
&= \inf_{\tilde{\Gamma}_\delta \in \tilde{\Gamma}_\Delta}
\inf_{x \in \R^{n-1}; |x|=1}
\frac{
x^T \left(
 \tilde{\Gamma}_c ^T \tilde{\Gamma}_c  
+
2  \tilde{\Gamma}_\delta ^T \tilde{\Gamma}_c 
+
 \tilde{\Gamma}_\delta ^T \tilde{\Gamma}_\delta 
\right) x
}{x^T x} \\
&\geq 
\min \{ \mu : \mu \in  \sigma( \tilde{\Gamma}_c^T \tilde{\Gamma}_c) \} 
-2 \| \tilde{\Gamma}_\Delta^T \tilde{\Gamma}_c \|
-  \| \tilde{\Gamma}_\Delta ^T \tilde{\Gamma}_\Delta \|
\end{align*}
As $\tilde{\Gamma}_c$ is an interval matrix with zero width, calculating $\min \{ \mu : \mu \in  \sigma( \tilde{\Gamma}_c^T \tilde{\Gamma}_c) \}$ can be done with  little error. 
%This bound on $\mu^*$  will suffice so long as $  \min \{ \mu : \mu \in  \sigma( \tilde{\Gamma}_c^T \tilde{\Gamma}_c) \}  > 
%2 \| \tilde{\Gamma}_\Delta^T \tilde{\Gamma}_c \|
%+  \| \tilde{\Gamma}_\Delta ^T \tilde{\Gamma}_\Delta \|$. 
 Thus, in the manners described above, we are able to estimate $\epsilon_b$ and $ \mu^*$. 
 
 Furthermore, recall that the matrices $\tilde{B}$ and $\tilde{\Gamma}$ were defined relative to $ U_{\hat{i}}$ for a choice of $ 1 \leq i \leq n$. 
%Recall that $ \bE_-^u(x)$ is an $n$-dimensional space containing $\mathrm{span}\{ U_{\varphi}\}$ where $ U_{\varphi}(x) = \{ \varphi_x , D \varphi_{xx}\}$. 
%In Section \ref{S:L_+} we selected as a basis  for $ \bE_-^u(x)$ the function $ U_{\varphi}$ and $ n-1$ other linearly independent functions $U_k$.  
Assuming that $U_{\varphi} \notin \text{span }\{  U_k \}$ for all  $1\leq k\leq n$, then there are $n$ different ways to define a frame matrix $\left[ U_{\varphi}| U_{\hat{i}} \right]$ and  compute corresponding    matrices $\tilde{B}$ and $\tilde{\Gamma}$.   
In practice, we optimize this choice of $i$  so that   $\epsilon_b = \frac{\|\tilde{B}\|  }{\sqrt{\mu^*}}$ is minimized.

\subsection{Counting Conjugate Points}
\label{sec:CountingConjugatePoints}

By Lemma \ref{lem:conj-det}, the problem of computing conjugate points is equivalent to counting the number of zeros of the function
\[
F(x) := \det \left[ \pi_1 \circ U(x) \right] . 
\]
In Section \ref{sec:CalculatingE-u} we discussed how to compute an interval enclosure of the  frame matrix $ U( x) \in  \cU( x) $   for $ \bE_-^u(x)$ and $x \in [-L_-, L_+]$. 
%Hence, we are able to compute an interval enclosure of the function $F$. 
%We may compute the interval enclosure $F^I(x) = \det \left[ \pi_1 \circ \cU^I(x) \right]$. 
In a straightforward manner, we are able to obtain interval bounds on $F$, and $F'$ as well by applying the Jacobi formula 
\[
\tfrac{d}{dx} \det U_1(x) = \text{tr}\left( \text{adj}(U_1(x)) U_1'(x)\right) . 
\]
\begin{Remark}
For $ x \leq -L_-$ we have bounds for a frame matrix of $ \mathbb{E}_-^u(x)$ given by  $ U(x) \in \bar{\cV}^{-,u} + \cE^-$ for $ \cE^-$ given in \eqref{eq:L_-Error}.
An alternative proof for Proposition \ref{prop:L-} showing $\bE_-^u(x) \cap \mathcal{D} = \{0\}$ may be obtained by directly computing $ \det \left( \pi_1 \circ (\bar{\cV}^{-,u} + \cE^-) \right)$ using interval arithmetic and showing this quantity is bounded away from $0$.
\end{Remark}

If $L_-$ and $L_+$ are chosen so that Propositions  \ref{prop:L-} and \ref{prop:main} are satisfied, then all of the zeros of $F$  are in the interval $ [ - L_-,L_+]$. 
To rigorously enumerate the zeros of $F(x)$, we take an approach based on the bisection method. 
That is,  we subdivide $ [-L_-,L_+]$ into a cover $\{ [x_i,x_{i+1}]\}$ of subintervals with non-overlapping interiors.  
For each subinterval $ [x_i,x_{i+1}]$,  we attempt to prove either that $F$ is bounded away from $0$, or $F$ has a unique zero in $(x_i,x_{i+1})$.
We do so by checking the hypotheses of the two statements  below:  
\begin{enumerate}
	\item If $ F(x_i) \cdot F(x_{i+1}) > 0$ (ie the endpoints have the same sign), and any of the following hold
	\begin{enumerate}
		\item $0 \notin F([x_i,x_{i+1}])$; ie direct computation.
		\item $0 \notin F'([x_i,x_{i+1}])$
		\item $ 0 \notin \left( F(x_i) + \left[0,\frac{x_{i+1} -x_i}{2}\right]    F'([x_i,x_{i+1}]) \right) \cup
		 \left( F(x_{i+1}) - \left[0,\frac{x_{i+1} -x_i}{2}\right]    F'([x_i,x_{i+1}]) \right) $ 
	\end{enumerate}   
	then there are no zeros of $F$ in $[x_i,x_{i+1}]$.  Note condition (c) is essentially the mean value theorem.  
	\item If $ F(x_i) \cdot F(x_{i+1}) < 0$ (ie the endpoints have different sign) and $0 \notin F'([x_i,x_{i+1}])$, then there is a unique zero of $F$ in $(x_i,x_{i+1})$. 
\end{enumerate}
If either statement (i) or (ii) holds on every subinterval in the cover $\{ [x_i,x_{i+1}]\}$, then $F$ has a number of zeros in $[-L_-,L_+]$ equal to the number of subintervals where statement $(ii)$ holds. 

If on any subinterval the computer program is unable to verify either statement (i) or (ii), then the computer assisted proof for enumerating the zeros of $ F$ on $[-L_-,L_+]$ is not successful. 
In such a case, the computation may be attempted again using a cover of $ [-L_- , L_+]$ with a finer mesh. This will generically improve the interval enclosure of estimates such as $F([x_i,x_{i+1}])$ or $F'([x_i,x_{i+1}])$.  
Further improvements may be had by reducing $L_+ + L_-$, because -- by the nature of rigorously solving an initial value problem -- as $x$ increases, the accuracy at which we are able to compute $F$ decreases.  
However, the value of $L_+ +L_-$ can only be reduced to the extent that we are still able to, as detailed in section \ref{sec:CalculationForL+}, prove that $\mathbb{E}_-^u(x) \cap \mathcal{D} = \{ 0\}$ for all $ x \geq L_+$.  

%%%%%%%%%%%%%%%%%%%%%%%%%%%%%%%%%%%%%%%%%%%%%%%%%%%%%%%%%%%%%

\section{(Un)stable Fronts in a Reaction-Diffusion System}\label{S:example1}

In this section we construct example PDEs of the form \eqref{E:gradrd}, use validated numerics to construct front solutions, and use the conjugate point method outlined above to show that fronts can be both stable and unstable. As noted in the introduction, this complements previous results. Classical Sturm-Liouville theory can be used to show that any front in a scalar reaction-diffusion equation must be stable (if the essential spectrum is stable), while any pulse in a scalar reaction-diffusion equation must be unstable. In \cite{BeckCoxJones18} it was shown that, for equations of the form \eqref{E:gradrd}, a symmetric pulse must also be unstable, thus in some sense generalizing the pulse instability result from the scalar to the system case. Here we show that the front stability result from the scalar case does not extend to the system case; fronts can be either stable or unstable. Although it is not surprising that both stable and unstable fronts can exists for systems of reaction-diffusion equation, we are not aware of this being discussed anywhere. 

We will build our example out of the scalar bistable equation
\begin{equation}\label{E:bistable}
u_t = u_{xx} + b f(u), \qquad f(u) =  u\left( u - \tfrac{1}{2}\right)(1-u), 
\end{equation}
which has an explicit front solution given by 
\begin{equation}\label{E:upcoupled-front}
\varphi_0(x; b) = \frac{1}{1+ e^{-x\sqrt{b/2}}}
\end{equation}
that is spectrally stable by the classical results of Sturm-Liouville theory. We note that the essential spectrum of the linearized operator is given by $(-\infty, -b/2]$, which is stable if $b > 0$. We will couple three copies of this equation together in such a way that the resulting system of equations has the gradient structure of \eqref{E:gradrd} so that all of the preceding results apply. To that end, consider the system
\begin{eqnarray} \label{eq:SpecificPDE}
\partial_t u_1 &=& \partial_x^2 u_1 +  b_1 f(u_1) + c_{12} g(u_1,u_2) \nonumber \\ 
\partial_t u_2 &=& \partial_x^2 u_2 +  b_2 f(u_2) + c_{12} g(u_2,u_1) + 
c_{23} g(u_2,u_3)
\label{E:ex1} \\
\partial_t u_3 &=& \partial_x^2 u_3 +  b_3 f(u_3) + c_{23} g(u_3,u_2), \nonumber 
\end{eqnarray}
where $f,g$ are given by 
\begin{align*}
f(u) &= u( u - \tfrac{1}{2})(1-u) 
&
g(u, w) &= w(1-w)(u- \tfrac{1}{2}).
\end{align*}
The coefficients $c_{ij}$ are the coupling coefficients; the positive parameters $b_i$ control the extent to which the uncoupled equations are the same. If the $c_{ij}$ are all zero, then there is a stationary front solution
\[
\Phi_0(x) = (\varphi_0(x; b_1),  \varphi_0(x; b_2),  \varphi_0(x; b_3)).
\]
Since each of the three components can be translated independently of the others, in this case the front is stable with a triple eigenvalue at the origin. The corresponding three independent eigenfunctions are given by
\[
(\varphi_0'(x; b_1),  0, 0), \qquad (0,  \varphi_0'(x; b_2),  0), \qquad (0,0,  \varphi_0'(x; b_3)).
\]
This system is of the form \eqref{E:gradrd} with $G: \R^3 \to \R$ given by
\begin{align} \label{eq:PotentialEnergy}
G(u_1,u_2,u_3) &= - \frac{1}{4} \sum_{1 \leq i \leq 3 }  b_iu_i^2 (1-u_i)^2 - \frac{1}{2} \sum_{1 \leq i \leq 2} c_{i, i+1} u_i (1 - u_i ) u_{i+1} ( 1 - u_{i+1}).
\end{align}
We will prove the following. 
\begin{Theorem}\label{thm:SpecificParameters}
Consider the PDE in \eqref{eq:SpecificPDE} with the parameter   $ b = (b_1,b_2,b_3)=(1,.98,.96)$ fixed. At each of the four parameter combinations $c_{\pm,\pm} = (\pm c_{12}, \pm c_{23}) = ( \pm .04, \pm .02)$,  there exists a standing wave solution  $\varphi_{\pm,\pm}$ to \eqref{eq:SpecificPDE} such that 
\begin{align*}
\lim_{x \to - \infty} \varphi_{\pm,\pm} (x) &= (0,0,0), &\lim_{x \to +\infty} \varphi_{\pm,\pm} (x) &= (1,1,1) .
\end{align*}
Furthermore,
\begin{itemize}
\item There are exactly 0 positive  eigenvalues in the point spectrum of   $\varphi_{-,-}$. 
\item There is exactly 1 positive  eigenvalue in the point spectrum of  both   $\varphi_{+,-}$ and $\varphi_{-,+}$. 
\item There are exactly 2 positive  eigenvalues in the point spectrum of   $\varphi_{+,+}$.
\end{itemize}
\end{Theorem}

We note that the specific parameters used in Theorem \ref{thm:SpecificParameters} are somewhat arbitrary, and the program used to produce the computer assisted proof  \cite{bib:codesConjugatePoints} is capable of performing the analogous stability calculation at other (individual) parameters. The parameters in Theorem \ref{thm:SpecificParameters} were selected so that the standing fronts $\varphi_{\pm,\pm}$ to the coupled equation could be computed using validated numerics, in particular when using the standing front to the uncoupled equation \eqref{E:bistable} as an initial approximation. Parameter continuation could be employed to investigate the existence of standing waves to \eqref{eq:SpecificPDE} and their spectral stability at parameters farther away from the uncoupled case, however this was not explored in this work. We also note that, even though we do not include it here, we did consider the corresponding example in the case $n = 2$, which produced similar results. Moreover, this example can be generalized in a relatively straightforward way to $n$ equations for $n > 3$.

\begin{proof}
The C++ code for the computer assisted proof of this proposition may be found in \cite{bib:codesConjugatePoints}. We summarize our approach. To calculate and prove the existence of these standing waves, it is equivalent to compute a connecting orbit to the ODE in \eqref{eq:HamiltonianODE}, which for our specific example is given by: 
	\begin{align*}
	u_1' &= v_1  ,
	&
	v_1' &= -b_1 f(u_1) - c_{12} g(u_1,u_2) \\
	%%%%%%%%%%%%%%%%%%%%%%%%%%%%%%
	u_2' &= v_2 ,
	&
	v_2' &= -b_2 f(u_2) - c_{12} g(u_2,u_1) - c_{23} g(u_2,u_3)\\
	%%%%%%%%%%%%%%%%%%%%%%%%%%%%%%
	u_3' &= v_3 ,
	&
	v_3' &= -b_3 f(u_3) - c_{23} g(u_3,u_2). 
	\end{align*}
	Define $p_0 = 0 \in \R^{6}$ and $ p_1 = ( 1, 1, 1, 0, 0, 0) \in \R^{6}$.  
	Then $ p_0,p_1$ are fixed points of \eqref{eq:HamiltonianODE} with equal energy $H(p_0)=H(p_1)=0$.  
	To  calculate and prove the existence of these standing waves, we use the methodology described in Section \ref{sec:ComputeStandingWave}. 
	
	For a fixed $L_- >0$, we use the methodology described in Section \ref{sec:L-Calculation} to calculate a rigorous enclosure of the three vectors which span  $\mathbb{E}^u_-(-L_-)$. 
	Furthermore, we show that 
	$\bE_{-}^{u}(x) \cap \cD = \{0\}$ for all $ x \leq -L_-$. 
	
	As described in Section \ref{sec:CalculatingE-u}, we may then solve the initial value problem  \eqref{E:main-sys},
	taking as initial conditions the three  vectors which span  $\mathbb{E}^u_-(-L_-)$, 
	and thus obtain a validated representation of  $\mathbb{E}^u_-(x)$ for $x \in [-L_-,L_+]$. 
	More precisely,  \eqref{E:main-sys} is given by
	$U_x = \tilde{\mathcal{A}} U$ for $  \tilde{\mathcal{A}}(x) = \left(\begin{smallmatrix} 0 & D^{-1} \\ -\mathcal{M}(x) & 0 \end{smallmatrix}\right)$ as in \eqref{E:main-sys} where for our example we take $D$ as the identity matrix, and the matrix $ \mathcal{M}(x) =\nabla^2 G(\varphi(x))$ is defined for the standing wave  $\varphi = (u_1,u_2,u_3)$  as
	\[
	\mathcal{M}(x) := 
	\begin{pmatrix}
	b_1 \tilde{f}(u_1) +c_{12} \tilde{g}(u_2) & c_{12} \tilde{h}(u_1,u_2) & 0 \\
	c_{12} \tilde{h}(u_1,u_2)  & b_2 \tilde{f}(u_2)+c_{12} \tilde{g}(u_1)+c_{23} \tilde{g}(u_3) & c_{23} \tilde{h}(u_2,u_3)  \\
	0 &  c_{23} \tilde{h}(u_2,u_3) & b_3 \tilde{f}(u_3)  + c_{23} \tilde{g}(u_2)
	\end{pmatrix}
	\]
	and  the functions $\tilde{f}, \tilde{g},\tilde{h}$ are defined as
	\begin{align*}
	\tilde{f}(u)&= \partial_u f(u) 
	&
	\tilde{g}(u)&= \partial_w g(u,w) 
	&
	\tilde{h}(u,w)&=  \partial_u g(u,w) 
	\\
	&= 3u(1-u) - \tfrac{1}{2},
	&
	&= u(1-u),
	&
	&= -2(u-\tfrac{1}{2})(w-\tfrac{1}{2}). 
	\end{align*}
	
	We then use the methodology described in Section \ref{sec:CalculationForL+} to verify the hypothesis of Proposition \ref{prop:main2}, and thus prove that 
	$\bE_{-}^{u}(x) \cap \cD = \{0\}$ for all $ x \geq L_+$. 
	
	Lastly, we use the techniques described in Section \ref{sec:CountingConjugatePoints} to count all of the conjugate points on the interval $[-L_-,L_+]$.  
	As we have also proved that there are no conjugate points when $ x \leq -L_-$ nor when $ x \geq L_+$, we have thus obtained a precise count of all conjugate points on $ x \in \R$. 
	By Theorem \ref{thm:prev-thm}, the number of conjugate points is equal to the number of positive  eigenvalues in the point spectrum of $\varphi_{\pm,\pm}$. 	
\end{proof}

For all values of $c = (c_{12}, c_{23})$ we consider, the front exists within an intersection of stable and unstable manifolds in the spatial dynamics. Each of these manifolds is $3$-dimensional, and the zero energy surface is codimension-1. Therefore, we generically expect a one-dimensional intersection. However, for $c=0$ the intersection is $3$-dimensional. This corresponds to the fact that, when $c = 0$, the equations for the $u_{i}$ are uncoupled, and as noted above each of the three component fronts can be translated independently of the other; this leads to the eigenvalue at $\lambda = 0$ of geometric and algebraic multiplicity three. For $c \neq 0$, one of the zero eigenvalues must remain at zero, but the others should move. One could compute the location of these other two eigenvalues by perturbing from the $c =0$ case. However, such a perturbative result would only hold for $|c|$ sufficiently small, without a precise characterization of what ``sufficiently small" means. One benefit to using validated numerics to obtain this (in)stability result is that we are not required to have $c$ near zero. 

%%%%%%%%%%%%%%%%%%%%%%%%%%%%%%%%%%%%%%%%%%%%%%%%%%%%%%%%%%%%%%%%%%%

\section{Future Directions}\label{S:future-directions}

One natural direction to explore next would be the extension of such results to a more general class of PDEs than \eqref{E:gradrd}. To do so, one would first need a theoretical result showing that the number of unstable eigenvalues can indeed be computed by instead counting conjugate points. For example, such a result for a class of fourth-order PDEs can be found in \cite{Howard20}. Next, one would need to develop a method using validated numerics, similar to what we have done above, to actually count those conjugate points. This is the subject of current work. A much more difficult, but in some sense more interesting, extension would be to extend these ideas to systems in multiple spatial dimensions by building upon the work in \cite{BeckCoxJones20, BeckCoxJones21, CastelliGameiroLessard18, CoxJonesLatushkin16,CoxJonesMarzuola15, DengJones11}.

In addition, as mentioned above the pulse instability result in \cite{BeckCoxJones18} applies only to {\it symmetric} pulses in systems of the form \eqref{E:gradrd}. It would be interesting to construct an example of an asymmetric pulse in such a system that is stable, thus implying, together with the example of \S\ref{S:example1}, that the stability/instability dichotomy for fronts/pulses in scalar reaction-diffusion equation is really specific to the scalar case. This is also the subject of current work.

\section{Acknowledgments}

The authors wish to thank the Mathematical Sciences Research Institute in Berkeley, CA; this work was begun while the authors
were in residence there during Fall 2018. M.B.\ would like to acknowledge the support of NSF grant DMS-1907923 and an AMS Birman Fellowship. The authors would like to thank Maciej Capi\'nski for showing us the ropes of the CAPD library while at MSRI.

%%%%

\bibliography{conj-pt-computation}
 
%%%%
 
\end{document}